\documentclass{article}
\usepackage{amsmath}
\usepackage{amssymb}
\usepackage{amsthm}
\usepackage{algorithm}
\usepackage{algorithmic}
\usepackage{graphicx}
\usepackage{color}
\usepackage{lscape}
\def\@themcountersep{}
\newtheorem{theorem}{Theorem}[section]

\newtheorem{corollary}[theorem]{Corollary}
\newtheorem{definition}[theorem]{Definition}

\newtheorem{lemma}[theorem]{Lemma}
\newtheorem{proposition}[theorem]{Proposition}

\def\ba{\begin{array}}
\def\ea{\end{array}}
\def\beq{\begin{equation}}
\def\eeq{\end{equation}}
\def\beann{\begin{eqnarray*} }
\def\eeann{\end{eqnarray*}}
\def\bc{\begin{center}}
\def\ec{\end{center}}
\def\bea{\begin{eqnarray}}
\def\eea{\end{eqnarray}}

\def\lag{\langle}
\def\rag{\rangle}

\def\com{{\rm com}\,}

\def\conv{{\rm conv}\,}

\def\Diag{{\rm Diag}\,}

\def\inter{{\rm int}\,}

\def\Tr{{\rm Tr}\,}

\begin{document}


\title{LP-based Tractable Subcones of the Semidefinite Plus Nonnegative Cone\thanks{The authors thank one of the anonymous reviewers for suggesting the title, which had previously been``Tractable Subcones and LP-based Algorithms for Testing Copositivity.'' 
This research was supported by the Japan Society for the Promotion of Science through a Grant-in-Aid for Scientific Research ((B)23310099) of the Ministry of Education, Culture, Sports, Science and Technology of Japan.}
}

\author{Akihiro Tanaka\thanks{
Graduate School of Systems and Information Engineering, University of Tsukuba, Tsukuba, Ibaraki 305-8573, Japan. email: a-tanaka@criepi.denken.or.jp
}   
 and 
Akiko Yoshise\thanks{Corresponding author. Faculty of Engineering, Information and Systems, University of Tsukuba, Tsukuba, Ibaraki 305-8573, Japan. email: yoshise@sk.tsukuba.ac.jp
}     
}

\date{December 2015 \\ Revised October 2017}

\maketitle

\begin{abstract}
The authors in a previous paper devised certain subcones of the semidefinite plus nonnegative cone and showed that satisfaction of the requirements for membership of those subcones can be detected by solving linear optimization problems (LPs) with $O(n)$ variables and $O(n^2)$ constraints. 
They also devised LP-based algorithms for testing copositivity using the subcones.
 In this paper, they investigate the properties of the subcones in more detail and explore larger subcones of the  positive semidefinite plus nonnegative cone whose satisfaction of the requirements for membership can be detected by solving LPs.
They introduce a {\em semidefinite basis (SD basis)} that is a basis of the space of $n \times n$ symmetric matrices consisting of $n(n+1)/2$ symmetric semidefinite matrices. Using the SD basis, they devise two new subcones for which detection can be done by solving LPs with $O(n^2)$ variables and $O(n^2)$ constraints. The new subcones are larger than the ones in the previous paper and inherit their nice properties. The authors also examine the efficiency of those subcones in numerical experiments. The results show that the subcones are promising for testing copositivity as a useful application.
\end{abstract}


\section{Introduction}
\label{sec:introduction}
Let $\mathcal{S}_n$ be the set of $n \times n$ symmetric matrices, and define their inner product as
\begin{equation}
\label{eq:inner_product}
\lag A, B \rag = \Tr (B^TA) = \sum_{i.j=1}^n a_{ij}b_{ij}.
\end{equation}
Bomze et al. \cite{aBOMZE00} coined the term ``copositive programming'' in relation to the following problem in 2000, on which many studies have since been conducted:
\[
\begin{array}{ll}
\mbox{Minimize} & \lag C, X \rag \\
\mbox{subject to} & \lag A_i, X \rag = b_i, \ (i=1,2,\ldots,m) \\
& X \in {\cal COP}_n.
\end{array}
\]
where $\mathcal{COP}_n$ is the set of $n \times n$ copositive matrices, i.e., matrices whose quadratic form takes nonnegative values on the $n$-dimensional nonnegative orthant $\mathbb{R}^n_+$:
\[
\mathcal{COP}_n := \{ X \in \mathcal{S}_n \mid d^T Xd \geq 0 \ \mbox{for all} \ d \in \mathbb{R}^n_+ \}.
\]
We call the set $\mathcal{COP}_n$ the {\em copositive cone}. A number of studies have focused on the close relationship between copositive programming and quadratic or combinatorial optimization (see, e.g., \cite{aBOMZE00}\cite{aBOMZE02}\cite{aDEKLERK02}\cite{aPOVH07}\cite{aPOVH09}\cite{BUNDFUSS09}\cite{aBURER09}\cite{aDICKINSON14b}). Interested readers may refer to \cite{aDUR10} and \cite{aBOMZE12} for background on and the history of copositive programming.

The following cones are attracting attention in the context of the relationship between combinatorial optimization and copositive optimization (see, e.g., \cite{aDUR10}\cite{aBOMZE12}).
Here, $\conv(S)$ denotes the convex hull of the set $S$.
\begin{description}
\item[-] The nonnegative cone ${\cal N}_n := \left\{X \in \mathcal{S}_n \mid x_{ij} \geq 0 \mbox{ for all} \ i, j \in \{1,2,\ldots,n\} \right\}$.
\item[-] The semidefinite cone ${\cal S}_n^+ := \{X \in \mathcal{S}_n \mid d^TXd \geq 0 \ \mbox{for all} \ d \in \mathbb{R}^n \} = \conv\left(\left\{xx^T \mid x \in \mathbb{R}^n \right\} \right) $.
\item[-] The copositive cone ${\cal COP}_n := \left\{X \in \mathcal{S}_n \mid d^TXd \geq 0 \ \mbox{ for all} \ d \in \mathbb{R}^n_+ \right\}$.
\item[-] The semidefinite plus nonnegative cone ${\cal S}_n^+ + {\cal N}_n$, which is the Minkowski sum  of ${\cal S}_n^+$ and ${\cal N}_n$.
\item[-] The union ${\cal S}_n^+ \cup {\cal N}_n$ of ${\cal S}_n^+$ and ${\cal N}_n$. 
\item[-] The doubly nonnegative cone ${\cal S}_n^+ \cap {\cal N}_n$, i.e., the set of positive semidefinite and componentwise nonnegative matrices.
\item[-] The completely positive cone ${\cal CP}_n := \conv\left(\left\{xx^T \mid x \in \mathbb{R}^n_+ \right\} \right) $.
\end{description}

Except the set ${\cal S}_n^+ \cup {\cal N}_n$, all of the above cones are proper (see Section 1.6 of \cite{bBERMAN03}, where a proper cone is called a {\em full cone}), and we can easily see from the definitions that the following inclusions hold:
\begin{equation}
\label{eq:inclusion2}
{\cal COP}_n \supseteq {\cal S}_n^+ + {\cal N}_n \supseteq {\cal S}_n^+ \cup {\cal N}_n \supseteq {\cal S}_n^+ \supseteq {\cal S}_n^+ \cap {\cal N}_n \supseteq {\cal CP}_n.
\end{equation}

While copositive programming has the potential of being a useful optimization technique, it still faces challenges. One of these challenges is to develop efficient algorithms for determining whether a given matrix is copositive. It has been shown that the above problem is co-NP-complete \cite{aMURTY87}\cite{aDICKINSON14}\cite{aDICKINSON14b} and many algorithms have been proposed to solve it (see, e.g., \cite{aBOMZE96}\cite{aBUNDFUSS08}\cite{aJOHNSON08}\cite{aJARRE09}\cite{aZILINSKAS11}\cite{aSPONSEL12}\cite{aBOMZE13}\cite{aDENG13}\cite{aDUR13}\cite{aTANAKA15}\cite{aBRAS15})
Here, we are interested in numerical algorithms which (a) apply to general symmetric matrices without any structural assumptions or dimensional restrictions and (b) are not merely recursive, i.e., do not rely on information taken from all principal submatrices, but rather focus on generating subproblems in a somehow data-driven way, as described in \cite{aBOMZE13}. There are few such algorithms, but they often use tractable subcones $\mathcal{M}_n$ of the semidefinite plus nonnegative cone ${\cal S}_n^+ + {\cal N}_n$ for detecting copositivity (see, e.g., \cite{aBUNDFUSS08}\cite{aSPONSEL12}\cite{aBOMZE13}\cite{aTANAKA15}). 
As described in Section \ref{sec:algorithms for copositivity}, these algorithms require one to check whether $A \in \mathcal{M}_n$ or $A \not\in \mathcal{M}_n$ repeatedly over simplicial partitions. The desirable properties of the subcones $\mathcal{M}_n \subseteq {\cal S}_n^+ + {\cal N}_n$ used by these algorithms can be summarized as follows:
\begin{description}
\item[P1] For any given $n \times n$ symmetric matrix $A \in \mathcal{S}_n$, we can check whether $A \in \mathcal{M}_n$ within a reasonable computation time, and
\item[P2] $\mathcal{M}_n$ is a subset of the semidefinite plus nonnegative cone ${\cal S}_n^+ + {\cal N}_n$ that at least includes the $n \times n$ nonnegative cone $\mathcal{N}_n$ and contains as many elements ${\cal S}_n^+ + {\cal N}_n$ as possible.
\end{description}
The authors, in \cite{aTANAKA15}, devised certain subcones of the semidefinite plus nonnegative cone ${\cal S}_n^+ + {\cal N}_n$ and showed that satisfaction of the requirements for membership of those cones can be detected by solving linear optimization problems (LPs) with $O(n)$ variables and $O(n^2)$ constraints. They also created an LP-based algorithm that uses these subcones for testing copositivity as an application of those cones.

The aim of this paper is twofold. First, we investigate the properties of the subcones in more detail, especially in terms of their convex hulls. Second, we search for subcones of the semidefinite plus nonnegative cone ${\cal S}_n^+ + {\cal N}_n$ that have properties {\bf P1} and {\bf P2}. To address the second aim, we introduce a {\em semidefinite basis (SD basis)} that is a basis of the space $\mathcal{S}_n$ consisting of $n(n+1)/2$ symmetric semidefinite matrices. Using the SD basis, we devise two new types of subcones for which detection can be done by solving LPs with $O(n^2)$ variables and $O(n^2)$ constraints. As we will show in Corollary \ref{coro:G and F}, these subcones are larger than the ones proposed in \cite{aTANAKA15} and inherit their nice properties. We also examine the efficiency of those subcones in numerical experiments.

This paper is organized as follows: In Section \ref{sec:subcones}, we show several tractable subcones of ${\cal S}_n^+ + {\cal N}_n$ that are receiving much attention in the field of copositive programming and investigate their properties, the results of which are summarized in Figures \ref {fig: subcones H G} and \ref{fig: Convex subcones H G}. In Section \ref{sec:SDbasis}, we propose new subcones of ${\cal S}_n^+ + {\cal N}_n$ having properties {\bf P1} and {\bf P2}. We define SD bases using Definitions \ref{def:SDbasisI} and \ref{def:SDbasisII} and construct new LPs for detecting whether a given matrix belongs to the subcones. 
In Section \ref{sec:algorithms for S+N}, we perform numerical experiments in which the new subcones are used for identifying the given matrices $A \in \mathcal{S}_n^+ + \mathcal{N}_n$.
As a useful application of the new subcones, Section \ref{sec:algorithms for copositivity} describes experiments for testing copositivity of matrices arising from the maximum clique problem and standard quadratic optimization problems. The results of these experiments show that the new subcones are promising not only for identification of $A \in \mathcal{S}_n^+ + \mathcal{N}_n$ but also for testing copositivity. We give concluding remarks in Section \ref{sec:concluding remarks}.

\section{Some tractable subcones of  ${\cal S}_n^+ + {\cal N}_n$ and related work}
\label{sec:subcones}

In this section, we show several tractable subcones of  the semidefinite plus nonnegative cone ${\cal S}_n^+ + {\cal N}_n$.
Since the set ${\cal S}_n^+ + {\cal N}_n$ is the dual cone of the doubly nonnegative cone ${\cal S}_n^+ \cap {\cal N}_n$,  we see that 
\begin{eqnarray*}
{\cal S}_n^+ + {\cal N}_n
& = & 
\{ A \in {\cal S}_n \mid \lag A, X \rag \geq 0 \ \mbox{for any} \ X \in  {\cal S}_n^+ \cap {\cal N}_n \} \\
& = & 
\{ A \in {\cal S}_n \mid \lag A, X \rag \geq 0 \ \mbox{for any} \ X \in  {\cal S}_n^+ \cap {\cal N}_n \ \mbox{such that} \ \Tr(X) = 1 \}
\end{eqnarray*}
and that the weak membership problem for ${\cal S}_n^+ + {\cal N}_n$ can be solved 
(to an accuracy of $\epsilon$) by solving the following doubly nonnegative program (which can be expressed as a semidefinite program of size $O(n^2)$).
\begin{equation}
\label{eq:SDP for S+N}
\begin{array}{ll}
\mbox{Minimize} & \lag A, X \rag \\
\mbox{subject to} & \lag I_n, X \rag = 1, \ X \in {\cal S}_n^+ \cap {\cal N}_n
\end{array}\end{equation}
where $I_n$ denotes the $n \times n$ identity matrix. Thus, the set ${\cal S}_n^+ + {\cal N}_n$ is a rather large and tractable convex subcone of $\mathcal{COP}_n$. However, solving the problem takes a lot of time \cite{aSPONSEL12}, \cite{aYOSHISE10} and does not make for a practical implementation in general. To overcome this drawback, more easily tractable subcones of ${\cal S}_n^+ + {\cal N}_n$ have been proposed.

We define  the matrix functions $N, S: {\cal S}_n \rightarrow {\cal S}_n$ such that, for $A \in \mathcal{S}_n$,  we have
\begin{equation}
\label{eq:N(A)}
N(A)_{ij}:=
\left\{
\begin{array}{ll}
A_{ij} & \ \  (A_{ij}>0 \ \mbox{and} \ i \neq j) \\
0 & \ \  (\mbox{otherwise})
\end{array}
\right.
\ \ \mbox{and} \ \ S(A):= A-N(A).
\end{equation}
In \cite{aSPONSEL12}, the authors defined the following set:
\begin{equation}
\label{eq:H_n}
\mathcal{H}_n := \{ A \in \mathcal{S}_n \mid S(A) \in \mathcal{S}_n^+ \}.
\end{equation}
Here, we should note that $A = S(A)+ N(A) \in \mathcal{S}_n^+ + \mathcal{N}_n$ if $A \in \mathcal{H}_n$. Also, for any $A \in \mathcal{N}_n$, $S(A)$ is a nonnegative diagonal matrix, and hence, $\mathcal{N}_n\subseteq \mathcal{H}_n$. 
The determination of $A \in \mathcal{H}_n$ is easy and can be done by extracting the positive elements $A_{ij} > 0  \ (i \neq j)$ as $N(A)_{ij}$ and by performing a Cholesky factorization of $S(A)$ (cf. Algorithm 4.2.4 in \cite{bGOLUB96}). Thus, from the inclusion relation (\ref{eq:inclusion2}), we see that the set $\mathcal{H}_n$ has the desirable {\bf P1} property. However, $S(A)$ is not necessarily positive semidefinite even if $A \in \mathcal{S}_n^+ + \mathcal{N}_n$ or $A \in \mathcal{S}_n^+$. The following theorem summarizes the properties of the set $\mathcal{H}_n$.
\begin{theorem}[\cite{aFIEDLER62} and Theorem 4.2 of \cite{aSPONSEL12}]
\label{theo:H_n}
$\mathcal{H}_n$ is a convex cone and $\mathcal{N}_n \subseteq \mathcal{H}_n \subseteq \mathcal{S}_n^+ + \mathcal{N}_n$. If $n \geq 3$, these inclusions are strict and ${\cal S}_n^+ \not\subseteq \mathcal{H}_n$. For $n=2$, we have $\mathcal{H}_n = \mathcal{S}_n^+ \cup \mathcal{N}_n = \mathcal{S}_n^+ + \mathcal{N}_n =\mathcal{COP}_n$.
\end{theorem}

The construction of the subcone $\mathcal{H}_n$ is based on the idea of  ``checking nonnegativity first and checking positive semidefiniteness second.'' In \cite{aTANAKA15}, another subcone is provided that is based on the idea of ``checking positive semidefiniteness first and checking nonnegativity second.'' 
Let ${\cal O}_n$ be the set of $n \times n$ orthogonal matrices and ${\cal D}_n$ be the set of $n \times n$ diagonal matrices.
For a given symmetric matrix $A \in \mathcal{S}_n$, suppose that $P=[p_1, p_2, \cdots, p_n] \in {\cal O}_n$ and $\Lambda = \Diag(\lambda_1, \lambda_2, \ldots, \lambda_n) \in {\cal D}_n$ satisfy
\begin{equation}
\label{eq:diagonalization}
A = P\Lambda P^T = \sum_{i=1}^n \lambda_i p_ip_i^T.
\end{equation}
By introducing another diagonal matrix $\Omega = \Diag(\omega_1, \omega_2, \ldots, \omega_n) \in {\cal D}_n$, we can make the following decomposition:
\begin{equation}
\label{eq:decomposition}
A= P(\Lambda-\Omega)P^T+P\Omega P^T
\end{equation}
If $\Lambda-\Omega \in {\cal N}_n$, i.e., if $\lambda_i \geq \omega_i \ (i = 1,2, \ldots,n)$, then the matrix $P(\Lambda-\Omega)P^T$ is positive semidefinite. Thus, if we can find a suitable diagonal matrix $\Omega \in {\cal D}_n$ satisfying
\begin{equation}
\label{eq:decomposition2}
\lambda_i \geq \omega_i \ (i = 1,2, \ldots,n), \ \ [P\Omega P^T]_{ij} \geq 0 \ (1 \leq i \leq j \leq n)
\end{equation}
then (\ref{eq:decomposition}) and (\ref{eq:inclusion2}) imply
\begin{equation}
\label{eq:decomposition3}
A= P(\Lambda-\Omega)P^T+P\Omega P^T \in \mathcal{S}_n^+ + \mathcal{N}_n 
\subseteq \mathcal{COP}_n.
\end{equation}
We can determine whether such a matrix exists or not by solving the following linear optimization problem with variables $\omega_i \ (i=1,2,\ldots,n) $ and $\alpha$:
\begin{equation}
\label{eq:LP}
\mbox{(LP)}_{P,\Lambda} \ 
\left|
\begin{array}{lll}
\mbox{Maximize} & \alpha & \\
\mbox{subject to} & \omega_i \leq \lambda_i & (i=1,2,\ldots,n) \\
& \displaystyle{ [P\Omega P^T]_{ij} = \left[\sum_{k=1}^n \omega_k p_kp_k^T \right]_{ij} \geq \alpha} & (1 \leq i \leq j \leq n) 
\end{array}
\right.
\end{equation}
Here, for a given matrix $A$,  $[A]_{ij}$ denotes the $(i,j)$th element of $A$.

Problem $\mbox{(LP)}_{P,\Lambda}$ has a feasible solution at which $\omega_i = \lambda_i \ (i=1,2,\ldots,n)$ and 
\[
\alpha 
= \min \left\{ \left[ P\Lambda P^T \right]_{ij} \mid 1 \leq i \leq j \leq n \right\}
= \min \left\{ \sum_{k=1}^n \lambda_k [p_k]_i[p_k]_j \mid 1 \leq i \leq j \leq n \right\}.
\]
For each $i=1,2,\ldots,n$, the constraints 
\[
[P\Omega P^T]_{ii} = \left[ \sum_{k=1}^n \omega_k p_kp_k^T \right]_{ii} = \sum_{k=1}^n \omega_k [p_k]_i^2 \geq \alpha
\]
and $\omega_k \leq \lambda_k \ (k=1,2,\ldots,n)$ imply the bound $\alpha \leq \min \left\{\sum_{k=1}^n \lambda_k [p_k]_i^2 \mid 1 \leq i \leq n \right\}$.\label{description:LP} Thus, $\mbox{(LP)}_{P,\Lambda}$ has an optimal solution with optimal value $\alpha_*(P, \Lambda)$. If $\alpha_*(P, \Lambda) \geq 0$, there exists a matrix $\Omega$ for which the decomposition (\ref{eq:decomposition2}) holds. The following set $\mathcal{G}_n$ is based on the above observations and was proposed in \cite{aTANAKA15} as the set,  $\mathcal{G}_n$
\begin{equation}
\label{eq:G_n}
\mathcal{G}_n :=\{A \in \mathcal{S}_n \mid \mathcal{PL}_{\mathcal{G}_n}(A) \neq \emptyset \}
\end{equation}
where 
\begin{equation}
\label{eq:PLG(A)}
\mathcal{PL}_{\mathcal{G}_n}(A) := \{ (P, \Lambda) \in {\cal O}_n \times {\cal D}_n \mid \mbox{$P$ and $\Lambda$ satisfy (\ref{eq:diagonalization}) and $\alpha_*(P,\Lambda) \geq 0$} \}
\end{equation}
for a given $A \in {\cal S}_n$.
As stated above, if $\alpha_*(P,\Lambda) \geq 0$ for a given decomposition $A = P\Lambda P^T$, we can determine $A \in \mathcal{G}_n$. In this case, we just need to compute a matrix decomposition and solve a linear optimization problem with $n+1$ variables and $O(n^2)$ constraints, which implies that it is rather practical to use the set $\mathcal{G}_n$ as an alternative to using  ${\cal S}_n^+ + {\cal N}_n$. Suppose that $A \in \mathcal{S}_n$ has $n$ different eigenvalues. Then the possible orthogonal matrices $P = [p_1, p_2, \cdots, p_n] \in {\cal O}_n$ are identifiable, except for the permutation and sign inversion of $\{p_1, p_2, \cdots, p_n \}$, and by representing (\ref{eq:diagonalization}) as 
\[
A = \sum_{i=1}^n \lambda_i p_ip_i^T,
\]
we can see that the problem $\mbox{(LP)}_{P,\Lambda}$ is unique for any possible $P \in {\cal O}_n$. In this case, $\alpha_*(P,\Lambda) < 0$ with a specific $P \in {\cal O}_n$ implies $A \not\in \mathcal{G}_n$. However, if this is not the case (i.e., an eigenspace of $A$ has at least dimension $2$), $\alpha_*(P,\Lambda) < 0$ with a specific $P \in {\cal O}_n$ does not necessarily guarantee that $A \not\in \mathcal{G}_n$. 

The above discussion can be extended to any matrix $P \in \mathbb{R}^{m \times n}$; i.e., it does not necessarily have to be orthogonal or even square. The reason why the orthogonal matrices $P \in  {\cal O}_n$ are dealt with here is that some decomposition methods for (\ref{eq:diagonalization}) have been established for such orthogonal $P$s.
The property $\mathcal{G}_n = \mbox{com}(\mathcal{S}_n, \mathcal{N}_n)$ in Theorem \ref{theo:G_n} also follows when $P$ is orthogonal. 

In \cite{aTANAKA15}, the authors described another set $\widehat{\mathcal{G}_n}$ that is closely related to $\mathcal{G}_n$.
\begin{equation}
\label{eq:G_n hat}
\widehat{\mathcal{G}_n}  :=  \{A \in \mathcal{S}_n \mid \mathcal{PL}_{\widehat{\mathcal{G}_n}}(A) \neq \emptyset \}
\end{equation}
where for $A \in {\cal S}_n$, the set $\mathcal{PL}_{\widehat{\mathcal{G}_n}}(A)$ is given by replacing ${\cal O}_n$ in (\ref{eq:PLG(A)}) by the space $\mathbb{R}^{n \times n}$ of $n \times n$ arbitrary matrices, i.e., 
\begin{equation}
\label{eq:PLGhat(A)}
\mathcal{PL}_{\widehat{\mathcal{G}_n}}(A) := \{ (P, \Lambda) \in \mathbb{R}^{n\times n} \times {\cal D}_n \mid \mbox{$P$ and $\Lambda$ satisfy (\ref{eq:diagonalization}) and $\alpha_*(P,\Lambda) \geq 0$} \}.
\end{equation}
If the set $\mathcal{PL}_{\mathcal{G}_n}(A)$ in (\ref{eq:PLG(A)}) is nonempty, then the set $\mathcal{PL}_{\widehat{\mathcal{G}_n}}(A)$ is also nonempty, which implies the following inclusions:
\begin{equation}
\label{eq:inclusion4}
\mathcal{G}_n
\subseteq
\widehat{\mathcal{G}_n}
\subseteq
\mathcal{S}_n^+ + \mathcal{N}_n.
\end{equation}

Before describing the properties of the sets $\mathcal{G}_n$ and $\widehat{\mathcal{G}_n}$, we will prove a preliminary lemma.

\begin{lemma}
\label{lemm: convex hull of two cones}
Let $\mathcal{K}_1$ and $\mathcal{K}_2$ be two convex cones containing the origin. 
Then $\conv(\mathcal{K}_1 \cup \mathcal{K}_2) = \mathcal{K}_1 + \mathcal{K}_2$.
\end{lemma}
\begin{proof}
Since $\mathcal{K}_1$ and $\mathcal{K}_2$ are convex cones, we can easily see that the inclusion $\mathcal{K}_1 + \mathcal{K}_2 \subseteq \conv(\mathcal{K}_1 \cup \mathcal{K}_2)$ holds. The converse inclusion also follows from the fact that $\mathcal{K}_1$ and $\mathcal{K}_2$ are convex cones. Since $\mathcal{K}_1$ and $\mathcal{K}_2$ contain the origin, we see that the inclusion $\mathcal{K}_1 \cup \mathcal{K}_2 \subseteq \mathcal{K}_1 + \mathcal{K}_2$ holds. From this inclusion and the convexity of the sets $\mathcal{K}_1$ and $\mathcal{K}_2$, we can conclude that 
\[
\conv(\mathcal{K}_1 \cup \mathcal{K}_2) \subseteq \conv(\mathcal{K}_1 + \mathcal{K}_2)
= \mathcal{K}_1 + \mathcal{K}_2.
\]
\end{proof}

The following theorem shows some of the properties of $\mathcal{G}_n$ and $\widehat{\mathcal{G}_n}$. Assertions (i) and (ii) were proved in Theorem 3.2 of \cite{aTANAKA15}. Assertion (iii) comes from the fact that $\mathcal{S}_n^+$ and  $\mathcal{N}_n$ are convex cones and from Lemma \ref{lemm: convex hull of two cones}.
Assertions (iv)-(vi) follow from (i)-(iii), the inclusion (\ref{eq:inclusion4}) and Theorem \ref{theo:H_n}.

\begin{theorem}
\label{theo:G_n}
\begin{description}
\item[(i)] 
$\mathcal{S}_n^+ \cup \mathcal{N}_n \subseteq \mathcal{G}_n$
\item[(ii)] $\mathcal{G}_n = \com(\mathcal{S}_n^+, \mathcal{N}_n)$, where the set $\com(\mathcal{S}_n^+, \mathcal{N}_n)$ is defined by
\[
\com(\mathcal{S}_n^+, \mathcal{N}_n) := 
\{ S + N \mid S \in \mathcal{S}_n^+, \ N \in \mathcal{N}_n, \ \mbox{\em $S$ and $N$ commute} \}.
\]
\item[(iii)]
$\conv(\mathcal{S}_n^+ \cup \mathcal{N}_n) = \mathcal{S}_n^+ + \mathcal{N}_n$.
\item[(iv)]
$
\mathcal{S}_n^+ \cup \mathcal{N}_n
\subseteq 
\mathcal{G}_n
=
\mbox{\em com}(\mathcal{S}_n^+, \mathcal{N}_n)
\subseteq
\widehat{\mathcal{G}_n}
\subseteq
\mathcal{S}_n^+ + \mathcal{N}_n
$.
\item[(v)] If $n = 2$, then
$
\mathcal{S}_n^+ \cup \mathcal{N}_n
=
\mathcal{G}_n
=
\com(\mathcal{S}_n^+, \mathcal{N}_n)
=
\widehat{\mathcal{G}_n}
=
\mathcal{S}_n^+ + \mathcal{N}_n
$
\item[(vi)]
$
\conv(\mathcal{S}_n^+ \cup \mathcal{N}_n)
=
\conv(\mathcal{G}_n)
=
\conv\left(\com(\mathcal{S}_n^+, \mathcal{N}_n)\right)
=
\conv(\widehat{\mathcal{G}_n})
=
\mathcal{S}_n^+ + \mathcal{N}_n
$.
\end{description}
\end{theorem}

A number of examples provided in \cite{aTANAKA15} illustrate the differences between $\mathcal{H}_n$, $\mathcal{G}_n$. Moreover, the following two matrices have three different eigenvalues, respectively, and we can identify 
\begin{equation}
\label{eq:examples}
\left[ \begin{array}{rrr} 2 & 2 & 2\\ 2 & 2 & -3\\ 2 & -3 & 6 \end{array} \right]
\in \mathcal{H}_3 \setminus  \mathcal{G}_3, \ \
\left[ \begin{array}{rrr} 1 & 5 & -2\\  5 & 1 & -2 \\ -2 & -2 & 4 \end{array} \right]
\in (\mathcal{S}_3^+ + \mathcal{N}_3) \setminus (\mathcal{H}_3 \cup \mathcal{G}_3)
\end{equation}
by solving the associated LPs. 
Figure \ref{fig: subcones H G} draws those examples and (ii) of Theorem \ref{theo:G_n}. 
Figure \ref{fig: Convex subcones H G} follows from (vii) of Theorem \ref{theo:G_n} and the convexity of the sets $\mathcal{N}_n$, $\mathcal{S}_n^+$ and $\mathcal{H}_n$ (see Theorem \ref{theo:H_n}).

\begin{figure}[ht]
\begin{center}
\includegraphics[width=130mm]{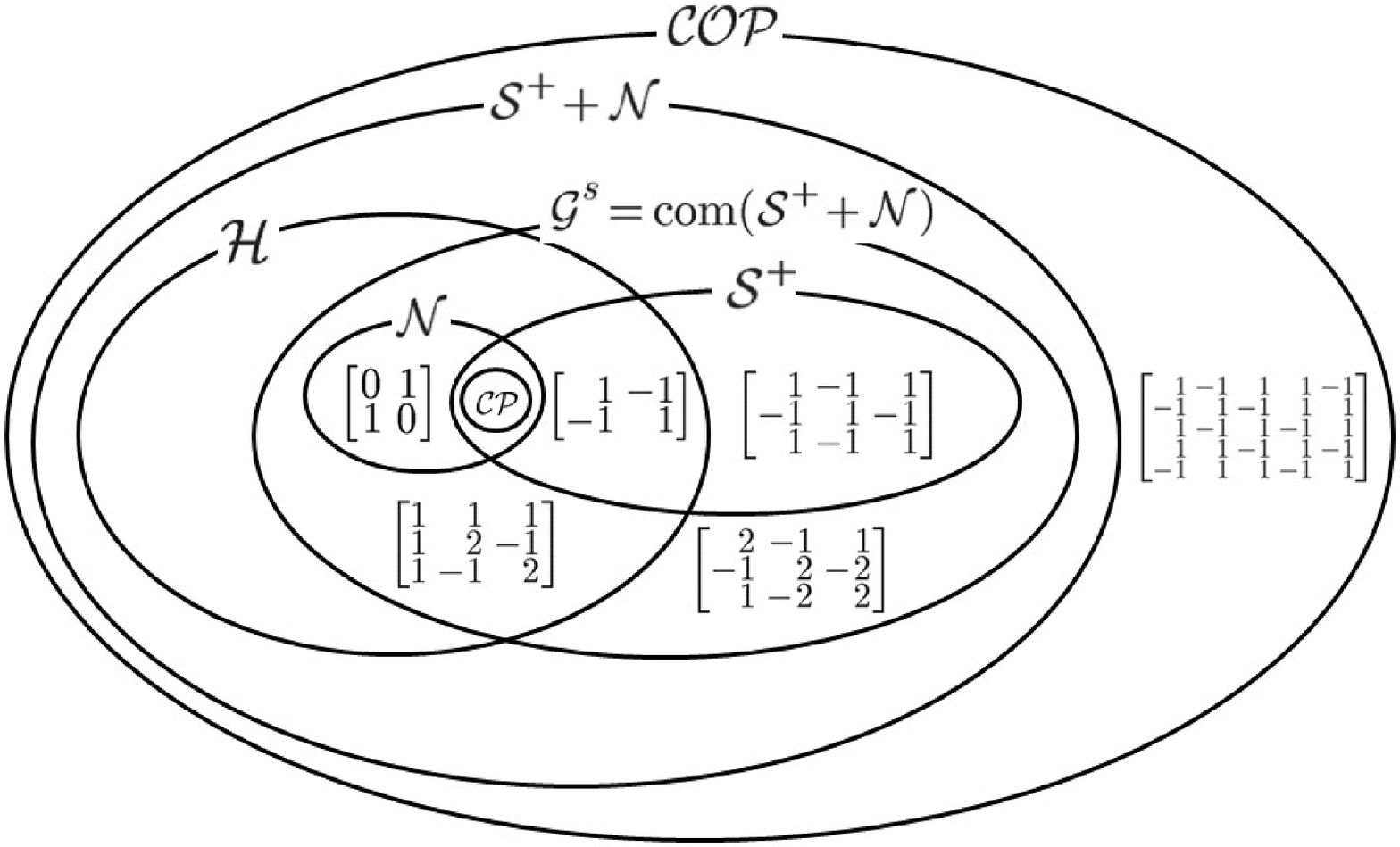}
\caption{Examples of inclusion relations among the subcones of $\mathcal{S}_n^+ + \mathcal{N}_n$ I}
\label{fig: subcones H G}
\end{center}
\end{figure}

\begin{figure}[ht]
\begin{center}
\includegraphics[width=130mm]{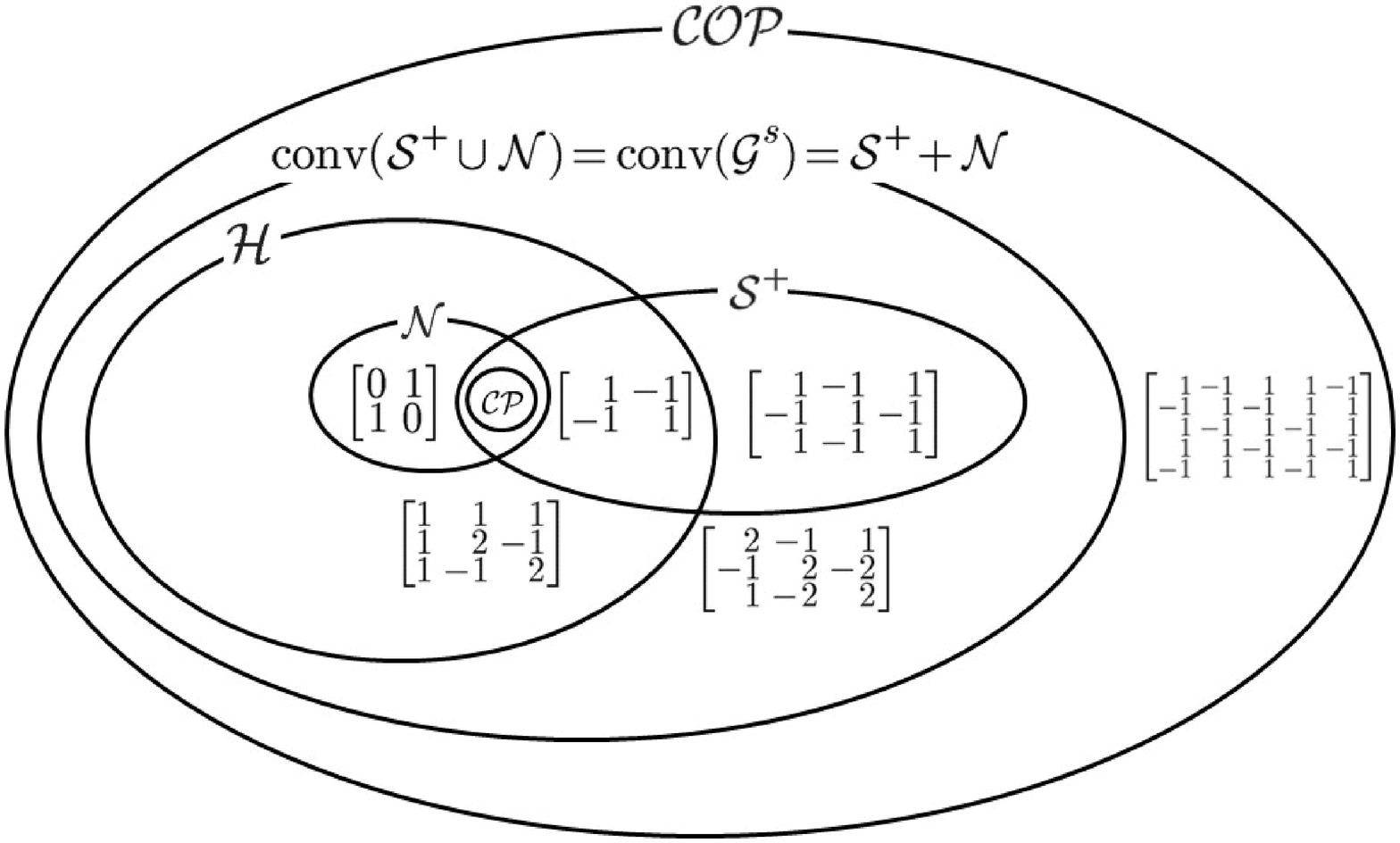}
\caption{Examples of inclusion relations among the subcones of $\mathcal{S}_n^+ + \mathcal{N}_n$ II}
\label{fig: Convex subcones H G}
\end{center}
\end{figure}

At present, it is not clear whether the set $\mathcal{G}_n = \com(\mathcal{S}_n^+, \mathcal{N}_n)$ is convex or not. As we will mention on page \pageref{page:nonconvex}, our numerical results suggest that the set might be not convex.

Before closing this discussion, we should point out another interesting subset of $\mathcal{S}_n^+ + \mathcal{N}_n$ proposed by Bomze and Eichfelder \cite{aBOMZE13}. Suppose that a given matrix $A \in \mathcal{S}_n$ can be decomposed as (\ref{eq:diagonalization}), and define the diagonal matrix $\Lambda_+$ by $ [\Lambda_+]_{ii} = \max\{0, \lambda_i \}$. Let $A_+: = P\Lambda_+P^T$ and $A_- := A_+ - A$. Then, we can easily see that $A_+$ and $A_-$ are positive semidefinite. Using this decomposition $A = A_+ - A_-$, Bomze and Eichfelder derived the following LP-based sufficient condition for $A \in \mathcal{S}_n^+ + \mathcal{N}_n$ in \cite{aBOMZE13}.

\begin{theorem}
[Theorem 2.6 of \cite{aBOMZE13}]
\label{theo:Bomze13}
Let $x \in \mathbb{R}_n^+$ be such that $A_+ x$ has only positive coordinates. If 
\[
(x^TA_+ x) (A_-)_{ii} \leq [(A_+x)_i ]^2 \ (i=1,2,\ldots,n)
\]
then $A \in \mathcal{COP}_n$.
\end{theorem}

Consider the following LP with $O(n)$ variables and $O(n)$ constraints,
\begin{equation}
\label{eq:LP_Bomze}
\inf\{ f^Tx \mid A_+ x \geq e, \ x \in \mathbb{R}_n^+ \}
\end{equation}
where $f$ is an arbitrary vector and $e$ denotes the vector of all ones. Define the set,
\[
\label{eq:L_n}
\mathcal{L}_n :=
\{ A \in \mathcal{S}_n \mid (x^TA_+ x) (A_-)_{ii} \leq [(A_+x)_i ]^2 \ (i=1,2,\ldots,n) \ \mbox{for some feasible solution $x$ of (\ref{eq:LP_Bomze})} \}.
\]

Then Theorem \ref{theo:Bomze13} ensures that $\mathcal{L}_n \subseteq \mathcal{COP}_n$. 
The following proposition gives a characterization when the feasible set of the LP of (\ref{eq:LP_Bomze}) is empty.

\begin{proposition}[Proposition 2.7 of \cite{aBOMZE13}]
\label{prop:LP_Bomze}
The condition ${\rm ker} A_+ \cap \{x \in \mathbb{R}_n^+ \mid e^Tx =1 \} \neq \emptyset$ is equivalent to $ \{x \in \mathbb{R}_n^+ \mid A_+ x \geq e\} = \emptyset$.
\end{proposition}

Consider the matrix, 
\[
A = \left[
\begin{array}{rr}
1&-1\\
-1&1\\
\end{array}
\right] 
\in \mathcal{S}_2^+.
\]
Thus, $A_+ = A$, and the set ${\rm ker} A_+ \cap \{x \in \mathbb{R}_n^+ \mid e^Tx =1 \} \neq \emptyset$. Proposition \ref{prop:LP_Bomze} ensures that $A \not\in \mathcal{L}_2$, and hence, $\mathcal{S}_n^+ \not \subseteq \mathcal{L}_n$ for $n \geq 2$, similarly to the set $\mathcal{H}_n$ for $n \geq 3$ (see Theorem \ref{theo:H_n}).

\section{Semidefinite bases}
\label{sec:SDbasis}
In this section, we improve the subcone $\mathcal{G}_n$ in terms of {\bf P2}. For a given matrix $A$ of (\ref{eq:diagonalization}), the linear optimization problem $\mbox{(LP)}_{P,\Lambda}$ in (\ref{eq:LP}) can be solved in order to find a nonnegative matrix that is a linear combination
\[
\sum_{i=1}^n \omega_i p_ip_i^T
\]
of $n$ linearly independent positive semidefinite matrices $p_ip_i^T \in \mathcal{S}_n^+ \ (i=1,2,\ldots,n)$. This is done by decomposing $A \in \mathcal{S}_n$ into two parts:
\begin{equation}
\label{eq:decomposition p}
A = \sum_{i=1}^n (\lambda_i - \omega_i) p_ip_i^T + \sum_{i=1}^n \omega_i p_ip_i^T
\end{equation}
such that the first part
\[
\sum_{i=1}^n (\lambda_i - \omega_i) p_ip_i^T
\]
is positive semidefinite. 
Since $p_ip_i^T \in \mathcal{S}_n^+ \ (i=1,2,\ldots,n)$ are only $n$ linearly independent matrices in $n(n+1)/2$ dimensional space $\mathcal{S}_n$, the intersection of the set of linear combinations of $p_ip_i^T$ and the nonnegative cone $\mathcal{N}_n$ may not have a nonzero volume even if it is nonempty. On the other hand, if we have a set of positive semidefinite matrices $p_ip_i^T \in \mathcal{S}_n^+ \ (i=1,2,\ldots,n(n+1)/2)$ that gives a basis of $\mathcal{S}_n$, then the corresponding intersection becomes the nonnegative cone $\mathcal{N}_n$ itself, and we may expect a greater  chance of finding a nonnegative matrix by enlarging the feasible region of $\mbox{(LP)}_{P,\Lambda}$. In fact, we can easily find a basis of $\mathcal{S}_n$ consisting of $n(n+1)/2$ semidefinite matrices from $n$ given orthogonal vectors $p_i \in \mathbb{R}^n \ (i=1,2,\ldots,n)$ based on the following result from \cite{aDICKINSON11}.

\begin{proposition}[Lemma 6.2 of \cite{aDICKINSON11}]
\label{theo:SDbasisI}
Let $v_i \in \mathbb{R}^n (i=1,2,\ldots,n)$ be $n$-dimensional linear independent vectors. Then the set
$\mathcal{V} := \{(v_i + v_j)(v_i + v_j)^T \mid 1 \leq i \leq j \leq n\}$ is a set of $n(n+1)/2$ linearly independent positive semidefinite matrices.
Therefore, the set $\mathcal{V}$ gives a basis of the set $\mathcal{S}_n$ of $n \times n$ symmetric matrices.
\end{proposition}

The above proposition ensures that the following set $\mathcal{B}_+(p_1,p_2,\ldots,p_n)$ is a basis of $n \times n$ symmetric matrices.

\begin{definition}
[Semidefinite basis type I]
\label{def:SDbasisI}
For a given set of $n$-dimensional orthogonal vectors $p_i \in \mathbb{R}^n (i=1,2,\ldots,n)$, define the map $\Pi_+: \mathbb{R}^n \times \mathbb{R}^n \rightarrow \mathcal{S}_n^+$ by
\begin{equation}
\label{eq:Pi type I}
\Pi_+(p_i, p_j) := \frac{1}{4}(p_i + p_j)(p_i + p_j)^T.
\end{equation}
We call the set 
\begin{equation}
\label{eq:SDPbasis+}
\mathcal{B}_+(p_1,p_2,\ldots,p_n) := \{ \Pi_+(p_i,p_j) \mid 1 \leq i \leq j \leq n \}
\end{equation}
a {\rm semidefinite basis type I} induced by $p_i \in \mathbb{R}^n (i=1,2,\ldots,n)$.
\end{definition}

A variant of the semidefinite basis type I is as follows. Noting that the equivalence 
\[
\Pi_+(p_i,p_j) = \frac{1}{2}p_ip_i^T+ \frac{1}{2}p_jp_j^T - \Pi_-(p_i,p_j)
\]
holds for any $i \neq j$, we see that $\mathcal{B}_-(p_1,p_2,\ldots,p_n)$ is also a basis of $n \times n$ symmetric matrices.

\begin{definition}
[Semidefinite basis type II]
\label{def:SDbasisII}
For a given set of $n$-dimensional orthogonal vectors $p_i \in \mathbb{R}^n (i=1,2,\ldots,n)$, define the map $\Pi_+: \mathbb{R}^n \times \mathbb{R}^n \rightarrow \mathcal{S}_n^+$ by
\begin{equation}
\label{eq:Pi type II}
\Pi_-(p_i, p_j) := \frac{1}{4}(p_i - p_j)(p_i - p_j)^T.
\end{equation}
We call the set 
\begin{equation} 
\label{eq:SDPbasis-}
\mathcal{B}_-(p_1,p_2,\cdots,p_n) := 
\{ \Pi_+(p_i,p_i) \mid 1 \leq i \leq n \} \cup
\{ \Pi_-(p_i,p_j) \mid 1 \leq i < j \leq n \}
\end{equation}
a {\em semidefinite basis type II} induced by $p_i \in \mathbb{R}^n (i=1,2,\ldots,n)$.
\end{definition}

Using the map $\Pi_+$ in (\ref{eq:Pi type I}), the linear optimization problem $\mbox{(LP)}_{P,\Lambda}$ in (\ref{eq:LP}) can be equivalently written as
\[
\mbox{(LP)}_{P,\Lambda} \ 
\left|
\begin{array}{lll}
\mbox{Maximize} & \alpha & \\
\mbox{subject to} & \omega_{ii}^+ \leq \lambda_i & (i=1,2,\ldots,n) \\
& \displaystyle{ \left[ \sum_{k=1}^n \omega_{kk}^+ \Pi_+(p_k,p_k) \right]_{ij} \geq \alpha} & (1 \leq i \leq j \leq n).
\end{array}
\right.
\]

The problem $\mbox{(LP)}_{P,\Lambda}$ is based on the decomposition (\ref{eq:decomposition p}). Starting with (\ref{eq:decomposition p}), the matrix $A$ can be decomposed using $\Pi_+(p_i,p_j)$ in (\ref{eq:Pi type I}) and $\Pi_-(p_i,p_j)$ in (\ref{eq:Pi type II}) as
\begin{eqnarray}
\nonumber
A & = &
\sum_{i=1}^n (\lambda_i-\omega^+_{ii}) \Pi_+(p_i, p_i) + \sum_{i=1}^n \omega^+_{ii} \Pi_+(p_i, p_i) \\
\nonumber
&= & 
\sum_{i=1}^n (\lambda_i-\omega^+_{ii}) \Pi_+(p_i, p_i) + \sum_{i=1}^n \omega^+_{ii} \Pi_+(p_i, p_i) \\
& & \ \ \
\label{eq:decomposition+}
+ \sum_{1 \leq i < j \leq n} (- \omega^+_{ij}) \Pi_+(p_i, p_j) + \sum_{1 \leq i < j \leq n} \omega^+_{ij} \Pi_+(p_i, p_j) \\
&= & 
\nonumber
\sum_{i=1}^n (\lambda_i-\omega^+_{ii}) \Pi_+(p_i, p_i) + \sum_{i=1}^n \omega^+_{ii} \Pi_+(p_i, p_i) \\
& & \ \ \
\nonumber
+ \sum_{1 \leq i < j \leq n} (- \omega^+_{ij}) \Pi_+(p_i, p_j) + \sum_{1 \leq i < j \leq n} \omega^+_{ij} \Pi_+(p_i, p_j) \\
& & \ \ \
\label{eq:decomposition+-}
+ \sum_{1 \leq i < j \leq n} (- \omega^-_{ij}) \Pi_-(p_i, p_j) + \sum_{1 \leq i < j \leq n} \omega^-_{ij} \Pi_-(p_i, p_j) .
\end{eqnarray}

On the basis of the decomposition (\ref{eq:decomposition+}) and (\ref{eq:decomposition+-}), we devise the following two linear optimization problems as extensions of $\mbox{(LP)}_{P,\Lambda}$:
\begin{eqnarray}
& & \label{eq:LP+} \
\mbox{(LP)}_{P,\Lambda}^+ \
\left|
\begin{array}{lll}
\mbox{Maximize} & \alpha & \\
\mbox{subject to} & \omega^+_{ii} \leq \lambda_i & (i=1,2,\ldots,n) \\
& \omega^+_{ij} \leq 0 & (1 \leq i < j \leq n) \\
& \displaystyle{ \left[ \sum_{1 \leq k \leq l \leq n} \omega^+_{kl} \Pi_+(p_k,p_l) \right]_{ij} \geq \alpha} & (1 \leq i \leq j \leq n)
\end{array} 
\right. \\
\nonumber
& & \\
\nonumber 
& & \mbox{(LP)}_{P,\Lambda}^{\pm} \
\left|
\begin{array}{lll}
\mbox{Maximize} & \alpha & \\
\mbox{subject to} & \omega^+_{ii} \leq \lambda_i & (i=1,2,\ldots,n) \\
& \omega^+_{ij} \leq 0, \ \omega^-_{ij} \leq 0 & (1 \leq i < j \leq n) \\
& \displaystyle{ \left[ \sum_{1 \leq k \leq l \leq n} \omega^+_{kl} \Pi_+(p_k,p_l) 
+ \sum_{1 \leq k < l \leq n} \omega^-_{kl} \Pi_-(p_k,p_l) \right]_{ij} \geq \alpha} & (1 \leq i \leq j \leq n) 
\end{array}
\right. \\
& & \label{eq:LP+-}
\end{eqnarray}

Problem $\mbox{(LP)}_{P,\Lambda}^+$ has $n(n+1)/2+1$ variables and $n(n+1)$ constraints, and problem $\mbox{(LP)}_{P,\Lambda}^{\pm}$ has $n^2 +1$ variables and $n(3n+1)/2$ constraints (see Table 1
). Since $[P\Omega P^T]_{ij}$ in (\ref{eq:LP}) is given by $\left[ \sum_{k=1}^n \omega_{kk} \Pi_+(p_k,p_k) \right]_{ij}$, we can prove that both linear optimization problems $\mbox{(LP)}_{P,\Lambda}^+$ and $\mbox{(LP)}_{P,\Lambda}^{\pm}$ are feasible and bounded by making arguments similar to the one for $\mbox{(LP)}_{P,\Lambda}$ on page \pageref{description:LP}. Thus, $\mbox{(LP)}_{P,\Lambda}^+$ and $\mbox{(LP)}_{P,\Lambda}^{\pm}$ have optimal solutions with corresponding optimal values $\alpha_*^+(P, \Lambda)$ and $\alpha_*^{\pm}(P, \Lambda)$.

If the optimal value $\alpha_*^+(P, \Lambda)$ of $\mbox{(LP)}_{P,\Lambda}^+$ is nonnegative, then, by rearranging (\ref{eq:decomposition+}), the optimal solution $\omega^{+*}_{ij} \ (1 \leq i \leq j \leq n)$ can be made to give the following decomposition:
\[
A = 
\left[ \sum_{i=1}^n (\lambda_i-\omega^{+*}_{ii}) \Pi_+(p_i, p_i) + \sum_{1 \leq i < j \leq n} (- \omega^{+*}_{ij}) \Pi_+(p_i, p_j) \right] 
+ 
\left[ \sum_{1 \leq i \leq j \leq n} \omega^{+*}_{ij} \Pi_+(p_i, p_j) \right] 
\in \mathcal{S}_n^+ + \mathcal{N}_n.
\]
In the same way, if the optimal value $\alpha_*^{\pm}(P, \Lambda)$ of $\mbox{(LP)}_{P,\Lambda}^{\pm}$ is nonnegative, then, by rearranging (\ref{eq:decomposition+-}), the optimal solution $\omega^{+*}_{ij} \ (1 \leq i \leq j \leq n)$, $\omega^{-*}_{ij} \ (1 \leq i < j \leq n)$ can be made to give the following decomposition: 
\begin{eqnarray*}
A & = &
\left[ \sum_{i=1}^n (\lambda_i-\omega^{+*}_{ii}) \Pi_+(p_i, p_i) + \sum_{1 \leq i < j \leq n} (- \omega^{+*}_{ij}) \Pi_+(p_i, p_j) + \sum_{1 \leq i < j \leq n} (- \omega^{-*}_{ij}) \Pi_-(p_i, p_j) \right] \\
& & \ \ \ \ \
+ 
\left[ \sum_{1 \leq i \leq j \leq n} \omega^{+*}_{ij} \Pi_+(p_i, p_j) 
+ \sum_{1 \leq i < j \leq n} \omega^{-*}_{ij} \Pi_-(p_i, p_j) \right] 
\ \ \in \mathcal{S}_n^+ + \mathcal{N}_n.
\end{eqnarray*}
On the basis of the above observations, we can define new subcones of $\mathcal{S}_n^+ + \mathcal{N}_n$ in a similar manner as (\ref{eq:G_n}) and (\ref{eq:G_n hat}).

For a given $A \in \mathcal{S}_n$, define the following four sets of pairs of matrices
\begin{equation}
\label{eq:PL(A)}
\begin{array}{rcl}
\mathcal{PL}_{\mathcal{F}_n^{+}}(A) & := & \{ (P, \Lambda) \in {\cal O}_n \times {\cal D}_n \mid \mbox{$P$ and $\Lambda$ satisfy (\ref{eq:diagonalization}) and $\alpha_*^{+}(P,\Lambda) \geq 0$} \} \\
\mathcal{PL}_{\mathcal{F}_n^{\pm}}(A) & := & \{ (P, \Lambda) \in {\cal O}_n \times {\cal D}_n \mid \mbox{$P$ and $\Lambda$ satisfy (\ref{eq:diagonalization}) and $\alpha_*^{\pm}(P,\Lambda) \geq 0$} \} \\
\mathcal{PL}_{\widehat{\mathcal{F}_n^{+}}}(A) & := & \{ (P, \Lambda) \in \mathbb{R}^{n\times n} \times {\cal D}_n \mid \mbox{$P$ and $\Lambda$ satisfy (\ref{eq:diagonalization}) and $\alpha_*^{+}(P,\Lambda) \geq 0$} \} \\
\mathcal{PL}_{\widehat{\mathcal{F}_n^{\pm}}}(A) & := & \{ (P, \Lambda) \in \mathbb{R}^{n\times n} \times {\cal D}_n \mid \mbox{$P$ and $\Lambda$ satisfy (\ref{eq:diagonalization}) and $\alpha_*^{\pm}(P,\Lambda) \geq 0$} \}
\end{array}
\end{equation}
where $\alpha_*^{+}(P,\Lambda)$ and $\alpha_*^{\pm}(P,\Lambda)$ are optimal values of $\mbox{(LP)}_{P,\Lambda}^+$ and $\mbox{(LP)}_{P,\Lambda}^{\pm}$, respectively.
Using the above sets, we define new subcones of  $\mathcal{S}_n^n + \mathcal{N}_n$ as follows:
\begin{equation}
\label{eq:F_n}
\renewcommand\arraystretch{1.3}
\begin{array}{rcl}
\mathcal{F}_n^{+} & := & \{A \in \mathcal{S}_n \mid \mathcal{PL}_{\mathcal{F}_n^{+}}(A) \neq \emptyset \}, \\
\mathcal{F}_n^{\pm} & := & \{A \in \mathcal{S}_n \mid \mathcal{PL}_{\mathcal{F}_n^{\pm}}(A) \neq \emptyset \}, \\
\widehat{\mathcal{F}_n^{+}} & := & \{A \in \mathcal{S}_n \mid \mathcal{PL}_{\widehat{\mathcal{F}_n^{+}}}(A) \neq \emptyset \}, \\
 \widehat{\mathcal{F}_n^{\pm}} & := &  \{A \in \mathcal{S}_n \mid \mathcal{PL}_{\widehat{\mathcal{F}_n^{\pm}}}(A) \neq \emptyset \}.
\end{array}
\renewcommand\arraystretch{\default}
\end{equation}

From the construction of problems $\mbox{(LP)}_{P,\Lambda}$, $\mbox{(LP)}_{P,\Lambda}^+$ and $\mbox{(LP)}_{P,\Lambda}^{\pm}$, and the definitions (\ref{eq:PL(A)}) and (\ref{eq:F_n}), we can easily see that
\[
\mathcal{G}_n
\subseteq \mathcal{F}_n^{+} 
\subseteq \mathcal{F}_n^{\pm}, \ \ \ \ \
\widehat{\mathcal{G}_n}
\subseteq \widehat{\mathcal{F}_n^{+}}
\subseteq  \widehat{\mathcal{F}_n^{\pm}}, \ \ \ \ \
\mathcal{F}_n^{+} \subseteq \widehat{\mathcal{F}_n^{+}}, \ \ \ \ \
\mathcal{F}_n^{\pm} \subseteq  \widehat{\mathcal{F}_n^{\pm}}
\]
hold. The corollary below follows from (iv)-(vi) of Theorem \ref{theo:G_n} and the above inclusions. 

\begin{corollary}
\label{coro:G and F}
\begin{description}
\item[(i)]
\[
\begin{array}{ccccccccc}
\mathcal{S}_n^+ \cup \mathcal{N}_n
& \subseteq &
\mathcal{G}_n = \mbox{\rm com}(\mathcal{S}_n^+ + \mathcal{N}_n)
& \subseteq &
\widehat{\mathcal{G}_n}
& \subseteq &
\mathcal{S}_n^+ + \mathcal{N}_n \\
& &
\rotatebox{270}{$\subseteq$}
& & 
\rotatebox{270}{$\subseteq$}
& & \\
\mathcal{S}_n^+ \cup \mathcal{N}_n
& \subseteq &
\mathcal{F}_n^{+} 
& \subseteq &
\widehat{\mathcal{F}_n^{+}}
& \subseteq &
\mathcal{S}_n^+ + \mathcal{N}_n \\
& &
\rotatebox{270}{$\subseteq$}
& & 
\rotatebox{270}{$\subseteq$}
& & \\
\mathcal{S}_n^+ \cup \mathcal{N}_n
& \subseteq &
\mathcal{F}_n^{\pm} 
& \subseteq &
 \widehat{\mathcal{F}_n^{\pm}}
& \subseteq &
\mathcal{S}_n^+ + \mathcal{N}_n \\
\end{array}
\]
\item[(ii)] If $n = 2$, then each of the sets $\mathcal{F}_n^{+}$, $\widehat{\mathcal{F}_n^{+}}$, $\mathcal{F}_n^{\pm}$, and $ \widehat{\mathcal{F}_n^{\pm}}$ coincides with $\mathcal{S}_n^+ + \mathcal{N}_n$.
\item[(iii)] The convex hull of each of the sets $\mathcal{F}_n^{+}$, $\widehat{\mathcal{F}_n^{+}}$,  $\mathcal{F}_n^{\pm}$, and $ \widehat{\mathcal{F}_n^{\pm}}$ is $\mathcal{S}_n^+ + \mathcal{N}_n$.
\end{description}
\end{corollary}

The following table summarizes the sizes of LPs (\ref{eq:LP}), (\ref{eq:LP+}), and (\ref{eq:LP+-}) that we have to solve in order to identify, respectively, 
$(P,\Lambda) \in \mathcal{PL}_{\mathcal{G}_n}(A)$ 
(or $(P,\Lambda) \in \mathcal{PL}_{\widehat{\mathcal{G}_n}}(A)$), 
$(P, \Lambda) \in \mathcal{PL}_{\mathcal{F}_n^{+}}(A)$
(or $(P, \Lambda)  \in \mathcal{PL}_{\widehat{\mathcal{F}_n^{+}}}(A)$), 
and $(P, \Lambda)  \in \mathcal{PL}_{\mathcal{F}_n^{\pm}}$ 
(or $(P, \Lambda)  \in\mathcal{PL}_{\widehat{\mathcal{F}_n^{\pm}}}(A)$).

\begin{table}[h]
\label{tab:LP}
\begin{center}
\caption{Sizes of LPs for identification}
\begin{tabular}{|c||c|c|c|}
\hline
\hline
\rule[0mm]{0mm}{5mm} 
Identification &$(P,\Lambda) \in \mathcal{PL}_{\mathcal{G}_n}(A)$  & $(P, \Lambda) \in \mathcal{PL}_{\mathcal{F}_n^{+}}(A)$ &  $(P, \Lambda)  \in \mathcal{PL}_{\mathcal{F}_n^{\pm}}$ \\
                 &(or $(P,\Lambda) \in \mathcal{PL}_{\widehat{\mathcal{G}_n}}(A)$) &(or $(P, \Lambda)  \in \mathcal{PL}_{\widehat{\mathcal{F}_n^{+}}}(A)$) &(or $(P, \Lambda)  \in\mathcal{PL}_{\widehat{\mathcal{F}_n^{\pm}}}(A)$) \\
\hline
\rule[0mm]{0mm}{4mm} 
\# of variables & $n+1$ & $n(n+1)/2 + 1$ & $n^2+1$ \\
\hline
\rule[0mm]{0mm}{4mm} 
\# of constraints & $n(n+3)/2$ & $n(n+1)$ & $n(3n+1)/2$ \\
\hline
\hline
\end{tabular}
\end{center}
\end{table}

\section{Identification of $A \in \mathcal{S}_n^+ + \mathcal{N}_n$}
\label{sec:algorithms for S+N}
In this section, we investigate the effect of using the sets  $\mathcal{G}_n$, $\mathcal{F}_n^{+}$ and $\mathcal{F}_n^{\pm}$ for identification of the fact $A \in \mathcal{S}_n^+ + \mathcal{N}_n$.

We generated random instances of $A \in \mathcal{S}_n^+ + \mathcal{N}_n$ by using the method described in Section 2 of \cite{aBOMZE13}. For an $n \times n$ matrix $B$ with entries independently drawn from a standard normal distribution, we obtained a random positive semidefinite matrix $S = BB^T$. An $n \times n$ random nonnegative matrix $N$ was constructed using $N = C - c_{\min} I_n$ with $C = F + F^T$ for a random matrix $F$ with entries uniformly distributed in $[0,1]$ and $c_{\min}$ being the minimal diagonal entry of $C$. We set $A = S + N \in \mathcal{S}_n^+ + \mathcal{N}_n$. The construction was designed so as to maintain the nonnegativity of $N$ while increasing the chance that $S + N$ would be indefinite and thereby avoid instances that are too easy.

For each instance $A \in \mathcal{S}^+_n+\mathcal{N}_n$, we used the MATLAB command ``$[P, \Lambda] = \mbox{eig} (A) $'' and obtained  $(P, \Lambda) \in \mathcal{O}_n \times \mathcal{D}_n$. 
We checked whether 
$(P, \lambda) \in \mathcal{PL}_{\mathcal{G}_n}$
 ($(P,L) \in \mathcal{PL}_{\mathcal{F}_n^{+}}$ and 
  $(P,L) \in \mathcal{PL}_{\mathcal{F}_n^{\pm}}$) 
by solving ($\mbox{LP})_{P,\Lambda}$ in (\ref{eq:LP}) ( ($\mbox{LP})_{P,\Lambda}^+$ in (\ref{eq:LP+}) and ($\mbox{LP})_{P,\Lambda}^{\pm}$ in (\ref{eq:LP+-}))
and if it held, we identified that $A \in \mathcal{G}_n$ ($A \in \mathcal{F}_n^{+}$ and $A \in \mathcal{F}_n^{\pm}$).

Table \ref{tab:identification} shows the number of matrices   (denoted by ``\#$A$'')  that were identified as $A \in \mathcal{H}_n$ ($A \in \mathcal{G}_n^{+}$, $A \in \mathcal{F}_n^{+}$, $A \in \mathcal{F}_n^{\pm}$ and $A \in \mathcal{S}_n^+ + \mathcal{N}_n$) and the average CPU time   (denoted by ``A.t.(s)''), where $1000$ matrices were generated for each $n$. 
We used a 3.07GHz Core i7 machine with 12 GB of RAM and Gurobi 6.5 for solving LPs.
Note that we performed the last identification  $A \in \mathcal{S}_n^+ + \mathcal{N}_n$ as a reference, while we used SeDuMi 1.3 with MATLAB R2015a for solving the semidefinite program (\ref{eq:SDP for S+N}).
The table yields the following observations:
\begin{itemize}
\item All of the matrices were identified as $A \in \mathcal{S}^+_n+\mathcal{N}_n$ by checking  $(P,L) \in \mathcal{PL}_{\mathcal{F}_n^{\pm}}$. The result is comparable to the one in Section 2 of \cite{aBOMZE13}.
The average CPU time for checking  $(P,L) \in \mathcal{PL}_{\mathcal{F}_n^{\pm}}$ is faster than the one for solving  the semidefinite program (\ref{eq:SDP for S+N}) when $n \geq 20$.
\item For any $n$, the number of identified matrices increases in the order of the set inclusion relation: $\mathcal{G}_n \subseteq \mathcal{F}_n^{+} \subseteq \mathcal{F}_n^{\pm}$, while the result for $\mathcal{H}_n \not \subseteq \mathcal{G}_n$ is better than the one for $\mathcal{G}_n$ when $n=10$.
\item For the sets $\mathcal{H}_n$, $\mathcal{G}_n$ and $\mathcal{F}_n^{+}$, the number of identified matrices decreases as the size of $n$ increases.
\end{itemize}

\begin{table}[h]
\caption{Results of identification of $A \in \mathcal{S}^+_n+\mathcal{N}_n$: $1000$ matrices were generated for each $n$.}
\label{tab:identification}
\begin{center}
\begin{tabular}{|c||r|r|r|r|r|r|r|r||r|r|}
\hline
\rule[0mm]{0mm}{5mm}&	\multicolumn{2}{c|}{$\mathcal{H}_n$} &	\multicolumn{2}{c|}{$\mathcal{G}_n$}&	\multicolumn{2}{c|}{$\mathcal{F}_n^{+}$}&	\multicolumn{2}{c||}{$\mathcal{F}_n^{\pm}$}&	\multicolumn{2}{c|}{$\mathcal{S}_n^+ + \mathcal{N}_n$}  \\
\cline{2-11}
\rule[0mm]{0mm}{5mm} $n$ &	\#$A$ & A.t.(s) &	\#$A$ & A.t.(s) & \#$A$ & A.t.(s) & \#$A$ & A.t.(s) & \#$A$ & A.t.(s)  \\
\hline
\hline
10&	791&0.001&		247&0.005& 		856&0.008&		1000&0.011&	1000&0.824\\
20&	16&0.001& 		20&0.013& 		719&0.121& 	1000&0.222&	1000&9.282\\
50&	0&0.003&		0&2.374& 		440&22.346&	1000&50.092&	1000&1285.981\\
\hline
\hline
\end{tabular}
\end{center}
\end{table}

\section{LP-based algorithms for testing $A \in \mathcal{COP}_n$}
\label{sec:algorithms for copositivity}
In this section, we investigate the effect of using the sets $\mathcal{F}_n^{+}$, $\widehat{\mathcal{F}_n^{+}}$, $\mathcal{F}_n^{\pm}$ and $\widehat{\mathcal{F}_n^{\pm}}$ for testing whether a given matrix $A$ is copositive by using Sponsel, Bundfuss, and D\"{u}r's algorithm \cite{aSPONSEL12}. 

\subsection{Outline of the algorithms}
\label{subsec:algorithm outline}
By defining the standard simplex $\Delta^S$ by $\Delta^S=\{x \in \mathbb{R}^n_+ \mid e^Tx =1 \}$, we can see that a given $n \times n$ symmetric matrix $A$ is copositive if and only if
\begin{eqnarray*}
x^TAx \geq 0 \ \mbox{ for all } \ x \in \Delta^S
\end{eqnarray*}
(see Lemma 1 of \cite{aBUNDFUSS08}). 
For an arbitrary simplex $\Delta$, 
a family of simplices $\mathcal{P}=\{\Delta^1, \ldots, \Delta^m\}$ is called a {\em simplicial partition} of $\Delta$ if it satisfies 
\begin{eqnarray*}
\Delta=\bigcup_{i=1}^m \Delta^i \ \mbox{and} \ 
\mbox{int}(\Delta^i)\cap\mbox{int}(\Delta^j) = \emptyset
\ \mbox{for all} \ i \neq j.
\end{eqnarray*}
Such a partition can be generated by successively bisecting simplices in the partition. For a given simplex $\Delta=\mbox{conv}\{v_1, \ldots, v_n\}$, consider the midpoint $v_{n+1}=\frac{1}{2}(v_i+v_j)$ of the edge $[v_i, v_j]$. Then the subdivision $\Delta^1=\{v_1, \ldots, v_{i-1}, v_{n+1}, v_{i+1}, \ldots, v_n\}$ and $\Delta^2=\{v_1, \ldots, v_{j-1}, v_{n+1}, v_{j+1}, \ldots, v_n\}$ of $\Delta$ satisfies the above conditions for simplicial partitions. See \cite{aHORST97} for a detailed description of simplicial partitions.

Denote the set of vertices of partition $\mathcal{P}$ by
\[
V(\mathcal{P})=\{v \mid \mbox{$v$ is a vertex of some $\Delta \in \mathcal{P}$} \}. 
\]
Each simplex $\Delta$ is determined by its vertices and can be represented by a matrix $V_\Delta$ whose columns are these vertices. Note that $V_\Delta$ is nonsingular and unique up to a permutation of its columns, which does not affect the argument \cite{aSPONSEL12}. Define the set of all matrices corresponding to simplices in partition $\mathcal{P}$ as
\[
M(\mathcal{P})=\{V_\Delta:\Delta \in \mathcal{P} \}.
\]
The ``fineness'' of a partition $\mathcal{P}$ is quantified by the maximum diameter of a simplex in $\mathcal{P}$, denoted by
\begin{equation}
\label{eq:delta(P)}
\delta(\mathcal{P})=\max_{\Delta \in \mathcal{P}}\max_{u, v \in \Delta}||u-v||.
\end{equation}

The above notation was used to show the following necessary and sufficient conditions for copositivity in \cite{aSPONSEL12}. The first theorem gives a sufficient condition for copositivity.

\begin{theorem}[Theorem 2.1 of \cite{aSPONSEL12}]
\label{theo:sufficient A in C}
If $A \in \mathcal{S}_n$ satisfies
\[
V^TAV \in \mathcal{COP}_n \ \mbox{for all} \ V \in M(\mathcal{P})
\]
then $A$ is copositive. Hence, for any $\mathcal{M}_n \subseteq \mathcal{COP}_n$, 
if $A \in \mathcal{S}_n$ satisfies
\[
V^TAV \in \mathcal{M}_n \ \mbox{for all} \ V \in M(\mathcal{P}),
\]
then $A$ is also copositive.
\end{theorem}

The above theorem implies that by choosing $\mathcal{M}_n = \mathcal{N}_n$ (see (\ref{eq:inclusion2})), $A$ is copositive if $V^T_\Delta AV_\Delta \in \mathcal{N}_n$ holds for any $\Delta \in \mathcal{P}$.

\begin{theorem}[Theorem 2.2 of \cite{aSPONSEL12}]
\label{theo:detect A in intC finitely}
Let $A \in \mathcal{S}_n$ be strictly copositive, i.e., $A \in \inter(\mathcal{COP}_n)$. Then there exists $\varepsilon > 0$ such that for all partitions $\mathcal{P}$ of $\Delta^S$ with $\delta(\mathcal{P}) < \varepsilon$, we have 
\[
V^TAV \in \mathcal{N}_n \ \mbox{for all} \ V \in M(\mathcal{P}).
\]
\end{theorem}

The above theorem ensures that if $A$ is strictly copositive (i.e., $A \in \inter(\mathcal{COP}_n)$), the copositivity of $A$ (i.e., $A \in \mathcal{COP}_n$) can be detected in finitely many iterations of an algorithm employing a subdivision rule with $\delta(\mathcal{P}) \rightarrow 0$. A similar result can be obtained for the case $A \not\in \mathcal{COP}_n$, as follows.

\begin{lemma}[Lemma 2.3 of \cite{aSPONSEL12}]
\label{lemm:detect A not in C finitely}
The following two statements are equivalent.
\begin{enumerate}
\item $A \notin \mathcal{COP}_n$
\item There is an $\varepsilon > 0$ such that for any partition $\mathcal{P}$ with
 $\delta(\mathcal{P})<\varepsilon$, there exists a vertex $v \in V(\mathcal{P})$ such that $v^TAv<0$.
\end{enumerate}
\end{lemma}

The following algorithm, from \cite{aSPONSEL12}, is based on the above three results. 

\begin{algorithm} 
\caption{Sponsel, Bundfuss, and D\"{u}r's algorithm to test copositivity} 
\label{alg:copositive1} 
\begin{algorithmic}[1] 
\REQUIRE $A \in \mathcal{S}_n, \mathcal{M}_n \subseteq \mathcal{COP}_n$
\ENSURE ``$A$ is copositive'' or ``$A$ is not copositive''
\STATE $\mathcal{P} \leftarrow \{\Delta^S\} $;
\WHILE{$\mathcal{P} \neq \emptyset$}
\STATE Choose $\Delta \in \mathcal{P}$;
\IF{$v^TAv < 0$ for some $v \in V(\{\Delta\})$:}
\label{line:not copositive1} 
\STATE \textbf{return} ``$A$ is not copositive'';
\ENDIF
\IF{we identify $V_\Delta^TAV_\Delta \in \mathcal{M}_n$}
\label{line:check M_n1}
\STATE $\mathcal{P} \leftarrow \mathcal{P} \setminus \{\Delta\}$;
\label{line:remove simplex1}
\ELSE
\STATE Partition $\Delta$ into $\Delta = \Delta^1 \cup \Delta^2$;
\label{line:refinement1}
\STATE $\mathcal{P} \leftarrow 
\mathcal{P}\setminus \{\Delta\} \cup \{\Delta^1, \Delta^2\}$;
\ENDIF
\ENDWHILE
\STATE \textbf{Return} ``$A$ is copositive'';
\end{algorithmic}
\end{algorithm}

As we have already observed, Theorem \ref{theo:detect A in intC finitely} and Lemma \ref{lemm:detect A not in C finitely} imply the following corollary.

\begin{corollary}
\label{coro:termination}
\begin{enumerate}
\item If $A$ is strictly copositive, i.e., $A \in \inter(\mathcal{COP}_n)$, then Algorithm \ref{alg:copositive1} terminates finitely, returning ``A is copositive.''
\item If $A$ is not copositive, i.e., $A \not\in \mathcal{COP}_n$, then Algorithm \ref{alg:copositive1} terminates finitely, returning ``A is not copositive.''
\end{enumerate}
\end{corollary}

In this section, we investigate the effect of using the sets $\mathcal{H}_n$ from (\ref{eq:H_n}), $\mathcal{G}_n$ from (\ref{eq:G_n}), and $\mathcal{F}_n^{+}$ and $\mathcal{F}_n^{\pm}$ from (\ref{eq:F_n}) as the set $\mathcal{M}_n$ in the above algorithm.

At Line \ref{line:check M_n1}, we can check whether $V_\Delta^TAV_\Delta \in \mathcal{M}_n$ directly in the case where $\mathcal{M}_n = \mathcal{H}_n$. 
In other cases, we diagonalize $V_\Delta^TAV_\Delta$ as $V_\Delta^TAV_\Delta = P \Lambda P^T$ and check whether $(P, \Lambda) \in \mathcal{PL}_{\mathcal{M}_n}(V_\Delta^TAV_\Delta)$ according to  definitions (\ref{eq:PLG(A)}) or (\ref{eq:PL(A)}).
If the associated LP has the nonnegative optimal value, then we identify $A \in \mathcal{M}_n$.

At Line \ref{line:remove simplex1}, Algorithm \ref{alg:copositive1} removes the simplex that was determined at Line \ref{line:check M_n1} to be in no further need of exploration by Theorem \ref{theo:sufficient A in C}. The accuracy and speed of the determination influence the total computational time and depend on the choice of the set $\mathcal{M}_n \subseteq \mathcal{COP}_n$.

Here, if we choose $\mathcal{M}_n = \mathcal{G}_n$ (respectively, $\mathcal{M}_n = \mathcal{F}_n^{+}$, $\mathcal{M}_n = \mathcal{F}_n^{\pm}$), we can improve Algorithm \ref{alg:copositive1} by incorporating the set $\widehat{\mathcal{M}_n} = \widehat{\mathcal{G}_n}$ (respectively, $\widehat{\mathcal{M}_n} = \widehat{\mathcal{F}_n^{+}}$, $\widehat{\mathcal{M}_n} = \widehat{\mathcal{F}_n^{\pm}}$), as proposed in \cite{aTANAKA15}.

\begin{algorithm}
\caption{Improved version of Algorithm 1} 
\label{alg:copositive2} 
\begin{algorithmic}[1] 
\REQUIRE $A \in \mathcal{S}_n, \mathcal{M}_n \subseteq \widehat{\mathcal{M}_n} \subseteq \mathcal{COP}_n$
\ENSURE ``$A$ is copositive'' or ``$A$ is not copositive''
\STATE $\mathcal{P} \leftarrow \{\Delta^S\} $;
\WHILE{$\mathcal{P} \neq \emptyset$}
\STATE Choose $\Delta \in \mathcal{P}$;
\IF{$v^TAv < 0$ for some $v \in V(\{\Delta\})$:}
\label{line:not copositive 2} 
\STATE \textbf{Return} ``$A$ is not copositive'';
\ENDIF
\STATE Let $P$ and $\Lambda$ be matrices satisfying $A = P\Lambda P^T$;
\IF{we identify $V_\Delta^TAV_\Delta \in \widehat{\mathcal{M}_n}$ by checking  whether $(V_{\Delta}^TP, \Lambda) \in \mathcal{PL}_{\widehat{\mathcal{M}_n}}$}
\label{line:check hat M_n1}
\STATE $\mathcal{P} \leftarrow \mathcal{P} \setminus \{\Delta\}$;
\ELSE
\STATE Let $P$ and $\Lambda$ be matrices satisfying $V_\Delta^TAV_\Delta=P\Lambda P^T$;
\IF{we identify $V_\Delta^TAV_\Delta \in \mathcal{M}_n$ by checking whether $(P, \Lambda) \in \mathcal{PL}_{\mathcal{M}_n}$}
\label{line:check M_n2}
\STATE $\mathcal{P} \leftarrow \mathcal{P} \setminus \{\Delta\}$;
\ELSE
\STATE Partition $\Delta$ into $\Delta = \Delta^1 \cup \Delta^2$, and set $\widehat{\Delta} \leftarrow \{\Delta^1, \Delta^2$\};
\label{line:refinement2} 
\FOR{$p = 1,2$}
\STATE Let $\Omega^* := \Diag(\omega^*)$ where $\omega^*$ is an LP optimal solution obtained at Line \ref{line:check M_n2};
\IF{we identify $V_{\Delta^p}^TAV_{\Delta^p} \in \widehat{\mathcal{M}_n}$ by checking whether $V_{\Delta^p}^TP\Omega^*P^TV_{\Delta^p} \in \mathcal{N}_n$}
\label{line:check hat M_n2}
\STATE $\widehat{\Delta} \leftarrow \widehat{\Delta} \setminus \{\Delta^p\}$;
\ENDIF 
\ENDFOR
\STATE $\mathcal{P} \leftarrow 
\mathcal{P}\setminus \{\Delta\} \cup \widehat{\Delta}$;
\ENDIF
\ENDIF
\ENDWHILE
\STATE \textbf{return} ``$A$ is copositive'';
\end{algorithmic}
\end{algorithm}

The details of the added steps are as follows. Suppose that we have a diagonalization of the form (\ref{eq:diagonalization}).

At Line \ref{line:check hat M_n1}, we need to solve an additional LP but do not need to diagonalize $V_\Delta^TAV_\Delta$. Let $P$ and $\Lambda$ be matrices satisfying (\ref{eq:diagonalization}). Then the matrix $V_{\Delta}^TP$ can be used to diagonalize $V_\Delta^TAV_\Delta$, i.e.,
\[
V_\Delta^TAV_\Delta 
= V_\Delta^T(P\Lambda P^T) V_{\Delta}
= (V_\Delta^TP) \Lambda(V_{\Delta}^TP)^T
\]
while $V_{\Delta}^TP \in \mathbb{R}^{n \times n}$ is not necessarily orthogonal. 
Thus, we can test whether $(V_{\Delta}^TP, \Lambda) \in \mathcal{PL}_{\widehat{\mathcal{M}_n}}$ by solving the corresponding LP according to the definitions (\ref{eq:PLGhat(A)}) or (\ref{eq:PL(A)}).
If $(V_{\Delta}^TP, \Lambda) \in \mathcal{PL}_{\widehat{\mathcal{M}_n}}$ holds, then we can identify $V_\Delta^TAV_\Delta \in \widehat{\mathcal{M}_n}$ 

If $(V_{\Delta}^TP, \Lambda) \not\in \mathcal{PL}_{\widehat{\mathcal{M}_n}}$ at Line \ref{line:check hat M_n1}, we proceed to the original step to identify whether $V_\Delta^TAV_\Delta \in \mathcal{M}_n$ at Line \ref{line:check M_n2}. Similarly to Line \ref{line:check M_n1} of Algorithm \ref{alg:copositive1}, we diagonalize $V_\Delta^TAV_\Delta$ as $V_\Delta^TAV_\Delta = P\Lambda P^T$ with an orthogonal matrix $P$ and a diagonal matrix $\Lambda$.
Then we check whether $(P, \Lambda) \in \mathcal{PL}_{\mathcal{M}_n}$ by solving the corresponding    LP, and if $(P, \Lambda) \in \mathcal{PL}_{\mathcal{M}_n}$, we can identify $V_\Delta^TAV_\Delta \in \mathcal{M}_n$.

At Line \ref{line:check hat M_n2}, we don't need to diagonalize $V_{\Delta^p}^TAV_{\Delta^p}$
 or solve any more LPs. Let $\omega^* \in \mathbb{R}^n$ be an optimal solution of the corresponding LP obtained at Line \ref{line:check hat M_n1} and let $\Omega^* := \Diag(\omega^*)$. Then the feasibility of $\omega^*$ implies the positive semidefiniteness of the matrix $V_{\Delta^p}^TP(\Lambda - \Omega^*)P^TV_{\Delta^p}$. Thus, if $V_{\Delta^p}^TP\Omega^*P^TV_{\Delta^p} \in \mathcal{N}_n$, we see that
\[
V_{\Delta^p}^TAV_{\Delta^p} 
= 
V_{\Delta^p}^TP(\Lambda - \Omega^*)P^TV_{\Delta^p}
+
V_{\Delta^p}^TP\Omega^*P^TV_{\Delta^p}
\in
\mathcal{S}_n^+ + \mathcal{N}_n
\]
and that $V_{\Delta^p}^TAV_{\Delta^p} \in \widehat{\mathcal{M}_n}$.

\subsection{Numerical results}
\label{subsec:numerical results}

This subsection describes experiments for testing copositivity using $\mathcal{N}_n$,$\mathcal{H}_n$, $\mathcal{G}_n$, $\mathcal{F}_n^{+}$, $\widehat{\mathcal{F}_n^{+}}$, $\mathcal{F}_n^{\pm}$ or $\widehat{\mathcal{F}_n^{\pm}}$ as the set $\mathcal{M}_n$ in Algorithms 1 and 2.
We implemented the following seven algorithms in MATLAB R2015a on a 3.07GHz Core i7 machine with 12 GB of RAM, using Gurobi 6.5 for solving LPs:
\begin{description}
\label{algorithms}
\item[{\bf Algorithm 1.1}:] Algorithm 1 with $\mathcal{M}_n = \mathcal{N}_n$.
\item[{\bf Algorithm 1.2}:] Algorithm 1 with $\mathcal{M}_n = \mathcal{H}_n$.
\item[{\bf Algorithm 2.1}:] Algorithm 2 with $\mathcal{M}_n = \mathcal{G}_n$ and $\widehat{\mathcal{M}_n} = \widehat{\mathcal{G}_n}$.
\item[{\bf Algorithm 1.3}:] Algorithm 1 with $\mathcal{M}_n = \mathcal{F}_n^{+}$.
\item[{\bf Algorithm 2.2}:] Algorithm 2 with $\mathcal{M}_n = \mathcal{F}_n^{+}$ and $\widehat{\mathcal{M}_n} = \widehat{\mathcal{F}_n^{+}}$.
\item[{\bf Algorithm 2.3}:] Algorithm 2 with $\mathcal{M}_n = \mathcal{F}_n^{\pm}$ and $\widehat{\mathcal{M}_n} = \widehat{\mathcal{F}_n^{\pm}}$.
\item[{\bf Algorithm 1.4}:] Algorithm 1 with $\mathcal{M}_n = \mathcal{S}_n^{+}+\mathcal{N}_n$.
\end{description}


As test instances, we used the two kinds of matrices arising from the maximum clique problem (Section \ref{subsubsec:max clique}) and from standard quadratic optimization problems (Section \ref{subsubsec:standard QP}).

\subsubsection{Results for the matrix arising from the maximum clique problem}
\label{subsubsec:max clique}

In this subsection, we consider the matrix
\begin{equation}
\label{eq:B_gamma}
B_{\gamma} := \gamma (E -A_G) -E
\end{equation}
where $E \in \mathcal{S}_n$ is the matrix whose elements are all ones and the matrix $A_G \in \mathcal{S}_n$ is the adjacency matrix of a given undirected graph $G$ with $n$ nodes. The matrix $B_\gamma$ comes from the maximum clique problem. The maximum clique problem is to find a clique (complete subgraph) of maximum cardinality in $G$. It has been shown (in \cite{aDEKLERK02}) that the maximum cardinality, the so-called clique number $\omega(G)$, is equal to the optimal value of
\[
\omega(G) = \min\{ \gamma \in \mathbb{N} \mid B_{\gamma} \in \mathcal{COP}_n \}.
\]
Thus, the clique number can be found by checking the copositivity of $B_{\gamma}$ for at most $\gamma=n,n-1, \ldots, 1$.

Figure \ref{fig:graph G} 
shows the instances of $G$ that were used in \cite{aSPONSEL12}. We know the clique numbers of $G_{8}$ and $G_{12}$ are $\omega(G_{8}) = 3$ and $\omega(G_{12}) = 4$, respectively.

\begin{figure}[h]
\begin{center}
\includegraphics[width=35mm]{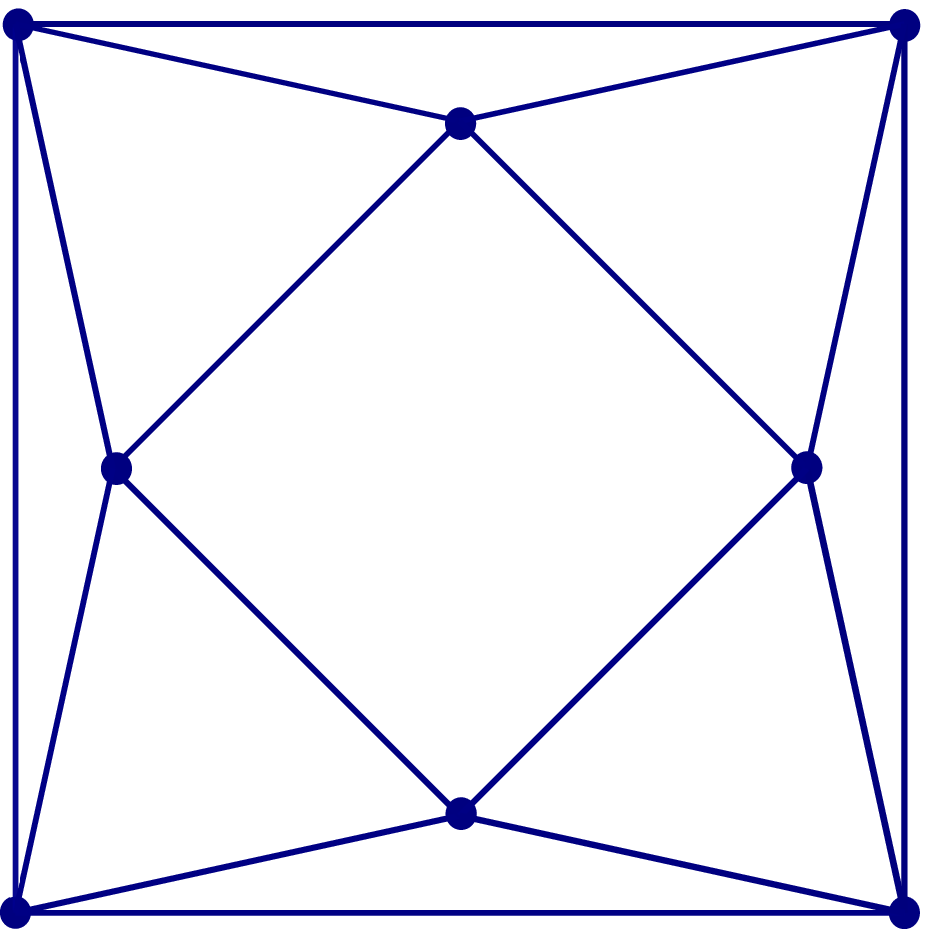}
\hspace{0.5cm}
\includegraphics[width=35mm]{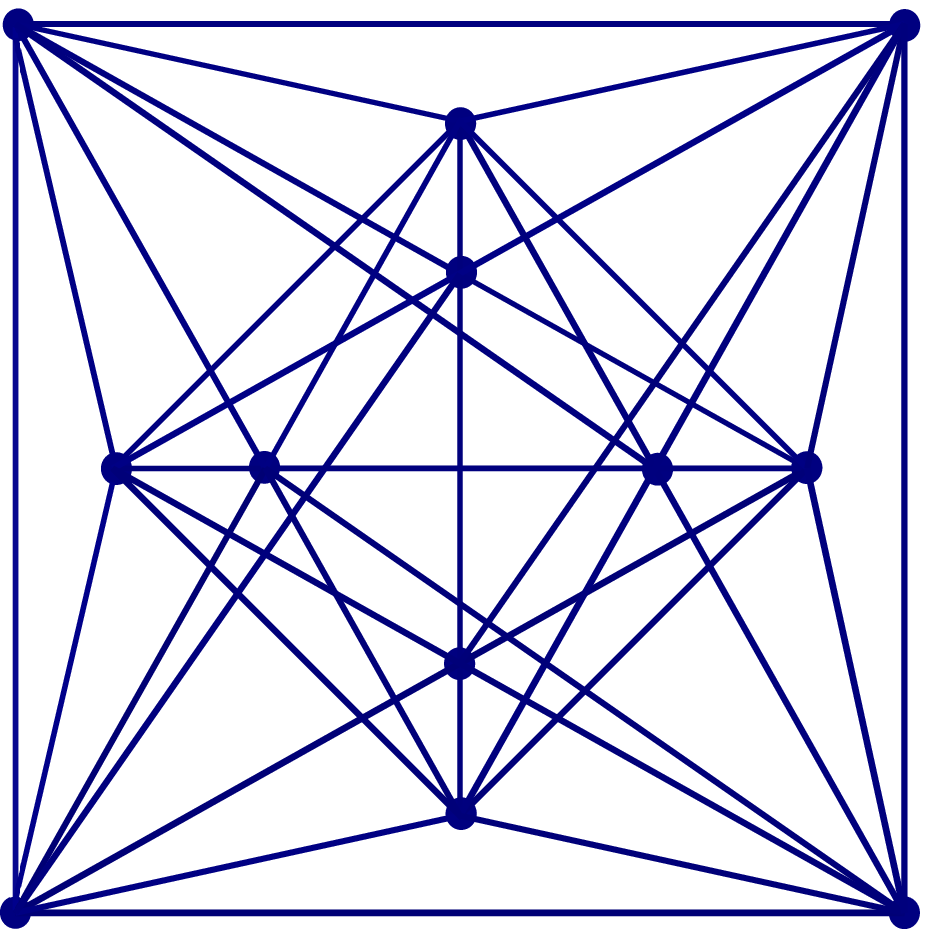}
\end{center}
\caption{Graphs $G_{8}$ with $\omega(G_{8})=3$ (left) and $G_{12}$ with $\omega(G_{12})=4$ (right).}
\label{fig:graph G}
\end{figure}

The aim of the implementation is to explore the differences in behavior when using $\mathcal{H}_n$, $\mathcal{G}_n$, $\mathcal{F}_n^{+}$, $\widehat{\mathcal{F}_n^{+}}$, $\mathcal{F}_n^{\pm}$ or $\widehat{\mathcal{F}_n^{\pm}}$ as the set $\mathcal{M}_n$ rather than to compute the clique number efficiently. Hence, the experiment examined $B_{\gamma}$ for various values of $\gamma$ at intervals of $0.1$ around the value $\omega(G)$ (see Tables \ref{tab:G8} and \ref{tab:G12} on page \pageref{tab:G8}).

As already mentioned, $\alpha_*(P,\Lambda) < 0$ ($\alpha^+_*(P,\Lambda) < 0$ and $\alpha^{\pm}_*(P,\Lambda) < 0$) with a specific $P$ does not necessarily guarantee that $A \not\in \mathcal{G}_n$ or $A \not\in \widehat{\mathcal{G}_n}$ ($A \not\in \mathcal{F}_n^{+}$ or $A \not\in \widehat{\mathcal{F}_n^{+}}$, $A \not\in \mathcal{F}_n^{\pm}$ or $A \not\in \widehat{\mathcal{F}_n^{\pm}}$). Thus, it not strictly accurate to say that we can use those sets for $\mathcal{M}_n$, and the algorithms may miss some of the $\Delta$'s that could otherwise have been removed. However, although this may have some effect on speed, it does not affect the termination of the algorithm, as it is guaranteed by the subdivision rule satisfying $\delta(\mathcal{P}) \rightarrow 0$, where $\delta(\mathcal{P})$ is defined by (\ref{eq:delta(P)}).

Tables \ref{tab:G8} and \ref{tab:G12} show the numerical results for $G_{8}$ and $G_{12}$, respectively. Both tables compare the results of 
the following seven algorithms in terms of the number of iterations (the column ``Iter.'') and the total computational time (the column ``Time (s)'' ):

The symbol ``$-$'' means that the algorithm did not terminate within 6 hours. The reason for the long computation time may come from the fact that for each graph $G$, the matrix $B_{\gamma}$ lies on the boundary of the copositive cone $\mathcal{COP}_n$ when $\gamma = \omega(G)$ ($\omega(G_8)=3$ and $\omega(G_{12})=4$).
See also Figure \ref{fig:graph G12v}, which shows a graph of the results of  {\bf Algorithms 1.2, 2.1, 2.3, and 1.4} for the graph $G_{12}$ in Table \ref{tab:G12}.

We can draw the following implications from the results in Table \ref{tab:G12} on page \pageref{tab:G12} for the larger graph $G_{12}$ (similar implications can be drawn from Table \ref{tab:G8}):
\begin{itemize}
\label{item:max clique}
\item At any $\gamma \geq 5.2$, {\bf Algorithms 2.1, 1,3, 2.2, 2.3, and 1.4} terminate in one iteration, and the execution times of  {\bf Algorithms 2.1, 1.3, 2.2, and 2.3} are much shorter than those of {\bf Algorithms 1.1, 1.2, or 1.4}.
\item The lower bound of $\gamma$ for which the algorithm terminates in one iteration and the one for which the algorithm terminates in 6 hours decrease in going from {\bf Algorithm 1.3} to {\bf Algorithm 3.1}. The reason may be that, as shown in Corollary \ref{coro:G and F}, the set inclusion relation $\mathcal{G}_n \subseteq \mathcal{F}_n^{+} \subseteq \mathcal{F}_n^{\pm} \subseteq  \mathcal{S}_n^{+}+\mathcal{N}_n$ holds.
\item Table 1 
 on page \pageref{tab:LP} summarizes the sizes of the LPs for identification. The results here imply that the computational times for solving an LP have the following magnitude relationship for any $n \geq 3$:
\[
\mbox{\bf Algorithm 2.1} < 
\mbox{\bf Algorithm 1.3} < 
\mbox{\bf Algorithm 2.2} < 
\mbox{\bf Algorithm 2.3} .
\]
On the other hand, the set inclusion relation $\mathcal{G}_n \subseteq \mathcal{F}_n^{+} \subseteq \mathcal{F}_n^{\pm}$ and the construction of Algorithms 1 and 2 imply that the detection abilities of the algorithms also follow the relationship described above and that the number of iterations has the reverse relationship for any $\gamma$s in Table \ref{tab:G12}:
\[
\mbox{\bf Algorithm 2.1} >
\mbox{\bf Algorithm 1.3} > 
\mbox{\bf Algorithm 2.2} >
\mbox{\bf Algorithm 2.3} .
\]
It seems that the order of the number of iterations has a stronger influence on the total computational time than the order of the computational times for solving an LP.
\item At each $\gamma \in [4.1, 4.9]$, the number of iterations of {\bf Algorithm 2.3} is much larger than one hundred times those of {\bf Algorithm 1.4}. This means that the total computational time of {\bf Algorithm 2.3} is longer than that of {\bf Algorithm 1.3} at each $\gamma \in [4.1, 4.9]$, while {\bf Algorithm 1.4} solves a semidefinite program of size $O(n^2)$ at each iteration.
\item At each $\gamma < 4$, the algorithms show no significant differences in terms of the number of iterations. The reason may be that they all work to find a $v \in V(\{\Delta\})$ such that $v^T(\gamma (E-A_G)-E)v<0$, while their computational time depends on the choice of simplex refinement strategy.
\end{itemize}

In view of the above observations, we conclude that Algorithm 2.3 with the choices $\mathcal{M}_n = \mathcal{F}_n^{\pm}$ and $\widehat{\mathcal{M}_n} = \widehat{\mathcal{F}_n^{\pm}}$ might be a way to check the copositivity of a given matrix $A$ when $A$ is strictly copositive.

The above results are in contrast with those of Bomze and Eichfelder in \cite{aBOMZE13}, where the authors show the number of iterations required by their algorithm for testing copositivity of matrices of the form (\ref{eq:B_gamma}).
On the contrary to the first observation described above, their algorithm terminates with few iterations when $\gamma < \omega(G)$, i.e., the corresponding matrix is not copositive, and it requires a huge number of iterations otherwise.

It should be noted that Table \ref{tab:G8} shows an interesting result concerning the non-convexity of the set $\mathcal{G}_n$, while we know that $\conv(\mathcal{G}_n) = \mathcal{S}_n^+ + \mathcal{N}_n$ (see Theorem \ref{theo:G_n}). \label{page:nonconvex}
 Let us look at the result at $\gamma = 4.0$ of {\bf Algorithm 2.1}. The multiple iterations at $\gamma = 4.0$ imply that we could not find $B_{4.0} \in \mathcal{G}_n$ at the first iteration for a certain orthogonal matrix $P$ satisfying (\ref{eq:diagonalization}). 
Recall that the matrix $B_{\gamma}$ is given by (\ref{eq:B_gamma}). It follows from $E - A_G \in \mathcal{N}_n \subseteq \mathcal{G}_n$ and from the result at $\gamma = 3.5$ in Table \ref{tab:G8} that
\[
0.5 (E - A_G) \in \mathcal{G}_n \ \mbox{and} \
B_{3,5} = 3.5 (E -A_G) -E \in \mathcal{G}_n.
\]
Thus, the fact that we could not determine whether the matrix
\[
B_{4.0} = 4.0 (E - A_G) - E = 0.5 (E - A_G) + B_{3.5}
\]
lies in the set $\mathcal{G}_n$ suggests that the set $\mathcal{G}_n = \mbox{com}(\mathcal{S}_n^+,  \mathcal{N}_n)$ is not convex.

\subsubsection{Results for the matrix arising from standard quadratic optimization problems}
\label{subsubsec:standard QP}


In this subsection, we consider the matrix
\begin{equation}
\label{eq:C_gamma}
C_{\gamma} := Q - \gamma E
\end{equation}
where $E \in \mathcal{S}_n$ is the matrix whose elements are all ones and $Q \in \mathcal{S}_n$ is an arbitrary symmetric matrix, not necessarily positive semidefinite.
The matrix $C_\gamma$ comes from standard quadratic optimization problems of the form, 
\begin{equation}
\begin{array}{ll}
\label{eq:QP}
\mbox{Minimize} & x^TQx \\
\mbox{subject to} & x \in \Delta^S := \{ x \in \mathbb{R}^n_+ \mid e^Tx = 1\}. \\
\end{array}
\end{equation}
In \cite{aBOMZE00}, it is shown that the optimal value of the problem 
\[
p^*(Q) = \max\{ \gamma \in \mathbb{R} \mid C_{\gamma} \in \mathcal{COP}_n \}.
\]
is equal to the optimal value of (\ref{eq:QP}).

The instances of the form (\ref{eq:QP}) were generated using the procedure {\em random\_qp} in \cite{NOWAK98} with two quartets of parameters $(n, s, k, d) = (10, 5, 5. 0.5)$ and  $(n, s, k, d) = (20, 10, 10. 0.5)$, where the parameter $n$ implies the size of $Q$, i.e., $Q$ is an $n \times n$ matrix.
It has been shown in \cite{NOWAK98} that {\em random\_qp} generates problems, for which we know the optimal value and a global minimizer a priori for each.
We set the optimal value as $-10$ for each quartet of parameters.

Tables \ref{tab:n10} and \ref{tab:n20} show the numerical results for $(n, s, k, d) = (10, 5, 5, 0.5)$  and $(n, s, k, d) = (20, 10, 10, 0.5)$. We generated $2$ instances for each quartet of parameters and performed the seven algorithms on page \pageref{algorithms} for these instances.
Both tables compare the average values of the seven algorithms in terms of the number of iterations (the column ``Iter.'') and the total computational time (the column ``Time (s)'' ):
the symbol ``$-$'' means that the algorithm did not terminate within 30 minutes. 
In each table, we made the interval between the values $\gamma$ smaller as $\gamma$ got closer to the optimal value, to observe the behavior around the optimal value more precisely. 

From the results in Tables  \ref{tab:n10} and \ref{tab:n20} on page \pageref{tab:n10}, 
we can draw implications that are very similar to those for the maximum clique problem, listed on page \pageref{item:max clique} (we hence, omitted discussing them here).
A major difference from the implications for the maximum clique problem is that {\bf Algorithm 1.2} using the set $\mathcal{H}_n$ is efficient for solving a small ($n=10$) standard quadratic problem, while it cannot solve the problem within 30 minutes when $n=20$ and $\gamma \geq -10.3125$.

\section{Concluding remarks}
\label{sec:concluding remarks}
In this paper, we investigated the properties of several tractable subcones of $\mathcal{S}_n^+ + \mathcal{N}_n$ and summarized the results (as Figures \ref {fig: subcones H G} and \ref{fig: Convex subcones H G}). We also devised new subcones of $\mathcal{S}_n^+ + \mathcal{N}_n$ by introducing the {\em semidefinite basis (SD basis) } defined as in Definitions \ref{def:SDbasisI} and \ref{def:SDbasisII}. We conducted numerical experiments using those subcones for identification of given matrices $A \in \mathcal{S}_n^+ + \mathcal{N}_n$ and for testing the copositivity of matrices arising from the maximum clique problem and from standard quadratic optimization problems. We have to solve LPs with $O(n^2)$ variables and $O(n^2)$ constraints in order to detect whether a given matrix belongs to those cones, and the computational cost is substantial. However, the numerical results shown in Tables \ref{tab:identification}, \ref{tab:G8}, \ref{tab:G12} and \ref{tab:n20} show that the new subcones are promising not only for identification of $A \in \mathcal{S}_n^+ + \mathcal{N}_n$ but also for testing copositivity.

Recently, Ahmadi, Dash and Hall \cite{AHMADI15} developed algorithms for inner approximating the cone of positive semidefinite matrices, wherein they focused on the set $\mathcal{D}_n \subseteq \mathcal{S}_n^+$ of $n \times n$ diagonal dominant matrices. Let $U_{n,k}$ be the set of vectors in $\mathbb{R}^n$ that have at most $k$ nonzero components, each equal to $\pm 1$, and define 
\[
\mathcal{U}_{n,k} := \{ uu^T \mid u \in U_{n,k} \}.
\]
Then, as the authors indicate, the following theorem has already been proven.
\begin{theorem}[Theorem 3.1 of \cite{AHMADI15}, Barker and Carlson \cite{aBARKER75}]
\[
\mathcal{D}_n = \mbox{\em cone}(\mathcal{U}_{n,k})
:= \left\{ \sum_{i=1}^{|\mathcal{U}_{n,k}|} \alpha_i U_i \mid U_i \in \mathcal{U}_{n,k}, \ \ \alpha_i \geq 0 \ (i=1, \ldots, |\mathcal{U}_{n,k}|) \right\}
\]
\end{theorem}
From the above theorem, we can see that for the SDP bases $\mathcal{B}_+(p_1, p_2, \cdots, p_n)$ in (\ref{eq:SDPbasis+}), $\mathcal{B}_-(p_1, p_2, \cdots, p_n)$ in (\ref{eq:SDPbasis-}) and $n$-dimensional unit vectors $e_1, e_2, \cdots, e_n$, the following set inclusion relation holds:
\[
\mathcal{B}_+(e_1, e_2, \cdots, e_n) \cup \mathcal{B}_-(e_1, e_2, \cdots, e_n) \subseteq \mathcal{D}_n = \mbox{cone}(\mathcal{U}_{n,k}).
\]
These sets should be investigated in the future. 

\section*{Acknowledgments}

The authors would like to sincerely thank the anonymous reviewers for their thoughtful and valuable comments which have significantly improved the paper.
Among others, one of the reviewers pointed out that Proposition \ref{theo:SDbasisI} is Lemma 6.2 of \cite{aDICKINSON11}, suggested to revise the title of the paper, and the definitions of the sets $\mathcal{G}_n$, $\mathcal{F}_n^{+}$, etc., to be more accurate, and gave the second example in (\ref{eq:examples}).

\begin{landscape}
\begin{table}[h]
\caption{Results for $B_{\gamma}$ with $G_{8}$}
\label{tab:G8}
\begin{center}
{\fontsize{6.4pt}{8pt}\selectfont
\begin{tabular}{|c||r|r|r|r|r|r|r|r|r|r|r|r|r|r|}
\hline
\hline
\rule[0mm]{0mm}{5mm}&	\multicolumn{2}{c|}{\bf Alg.~1.1~{$(\mathcal{N}_n)$}} &	\multicolumn{2}{c|}{\bf Alg.~1.2~{($\mathcal{H}_n$)}}&	\multicolumn{2}{c|}{\bf Alg.~2.1~{($\mathcal{G}_n$ , $\widehat{\mathcal{G}_n}$)}}& \multicolumn{2}{c|}{\bf Alg.~1.3~{($\mathcal{F}_n^{+}$)}}& \multicolumn{2}{c|}{\bf Alg.~2.2~{($\mathcal{F}_n^{+}$ , $\widehat{\mathcal{F}_n^{+}}$)}}&\multicolumn{2}{c|}{\bf Alg.~2.3~{($\mathcal{F}_n^{\pm}$ , $\widehat{\mathcal{F}_n^{\pm}}$)}}&\multicolumn{2}{c|}{\bf Alg.~1.4~{($\mathcal{S}_n^+$ + $\mathcal{N}_n$)}}\\
\cline{2-15}
\rule[0mm]{0mm}{5mm} $\gamma$&	\multicolumn{1}{c|}{Iter.}&	\multicolumn{1}{c|}{Time(s)}&	\multicolumn{1}{c|}{Iter.}& \multicolumn{1}{c|}{Time(s)}& \multicolumn{1}{c|}{Iter.}&	\multicolumn{1}{c|}{Time(s)}&	\multicolumn{1}{c|}{Iter.}&	\multicolumn{1}{c|}{Time(s)}&	\multicolumn{1}{c|}{Iter.}&	\multicolumn{1}{c|}{Time(s)} &	\multicolumn{1}{c|}{Iter.}&	\multicolumn{1}{c|}{Time(s)}&	\multicolumn{1}{c|}{Iter.}&\multicolumn{1}{c|}{Time(s)}\\
\hline
\hline
1.5&		1&0.001&			1&0.001&				1&0.004&				1&0.004&				1&0.006&				1&0.007&			1&0.177\\
1.6&		1&0.001&			1&0.001&				1&0.004&				1&0.004&				1&0.006&				1&0.008&			1&0.139\\
1.7&		1&0.001&			1&0.001&				1&0.004&				1&0.004&				1&0.006&				1&0.007&			1&0.180\\
1.8&		1&0.001&			1&0.001&				1&0.004&				1&0.004&				1&0.006&				1&0.007&			1&0.151\\
1.9&		1&0.001&			1&0.001&				1&0.004&				1&0.004&				1&0.006&				1&0.007&			1&0.225\\
2.0&		159&0.018&		159&0.021&			159&0.454&			159&0.368&			159&0.690&			159&0.964&		159&31.782\\
2.1&		159&0.019&		159&0.020&			159&0.455&			159&0.362&			159&0.689&			159&0.976&		159&31.237\\
2.2&		159&0.019&		159&0.020&			159&0.454&			159&0.363&			159&0.698&			159&0.971&		159&31.352\\
2.3&		159&0.019&		159&0.020&			159&0.452&			159&0.364&			159&0.689&			159&0.964&		159&30.115\\
2.4&		159&0.019&		159&0.020&			159&0.450&			159&0.362&			159&0.685&			159&0.957&		159&30.354\\
2.5&		159&0.019&		159&0.020&			159&0.453&			159&0.364&			159&0.690&			159&0.964&		159&29.681\\
2.6&		159&0.019&		159&0.020&			159&0.453&			159&0.361&			159&0.684&			159&0.953&		159&30.060\\
2.7&		2942&0.495&		2422&0.346&			2687&7.313&			2553&5.623&			2461&9.772&			2307&12.480&		1613&293.441\\
2.8&		2942&0.492&		2246&0.301&			2463&7.197&			1951&4.524&			1811&7.355&			1635&8.731&		1251&448.201\\
2.9& 	2942&0.501&		1606&0.191&			2139&6.270&			1493&3.469&			1393&5.458&			1309&6.867&		1251&449.572\\
3.0&		-&-&			-&-&				-&-& 				-&-&				-&-& 				-&-& 			-&-\\
3.1&		263548&261.285&			3003&0.279&			5885&14.603&			1827&3.864&			1357&4.879&			503&2.394&		7&3.186\\
3.2&		255202&243.819&			1509&0.132&			3129&7.830&			911&1.980&			377&1.347&			201&0.976&		3&1.480\\
3.3&		70814&24.332&			469&0.040&			2229&5.549&			447&0.968&			249&0.918&			111&0.538&		3&1.352\\
3.4&		70814&23.735&			395&0.034&			1603&4.112&			291&0.625&			167&0.650&			53&0.254	&		3&1.401\\
3.5&		70814&23.821&			369&0.031&			1&0.003&				1&0.003&				1&0.004&				1&0.004&			1&0.322\\
3.6&		70814&24.304&			209&0.017&			1&0.002&				1&0.003&				1&0.004&				1&0.004&			1&0.362\\
3.7&		70814&24.302&			115&0.009&			1&0.002&				1&0.003&				1&0.004&				1&0.004&			1&0.371\\
3.8&		70814&23.744&			79&0.007&			1&0.002&				1&0.003&				1&0.004&				1&0.004&			1&0.359\\
3.9&		70814&24.101&			63&0.005&			1&0.002&				1&0.003&				1&0.003&				1&0.005&			1&0.322\\
4.0&		70814&24.242&			47&0.004&			227&0.593&			1&0.003&				1&0.003&				1&0.005&			1&0.360\\
4.1&		4660&0.431&			23&0.002&			1&0.003&				1&0.003&				1&0.003&				1&0.005&			1&0.324\\
4.2&		4660&0.434&			17&0.002&			1&0.005&				1&0.003&				1&0.003&				1&0.005&			1&0.330\\
4.3&		4660&0.432&			17&0.002&			1&0.005&				1&0.003&				1&0.003&				1&0.005&			1&0.324\\
4.4&		4660&0.433&			7&0.001&				1&0.005&				1&0.003&				1&0.003&				1&0.005&			1&0.328\\
4.5&		4660&0.435&			7&0.001&				1&0.005&				1&0.003&				1&0.003&				1&0.006&			1&0.258\\
\hline
\hline
\end{tabular}
}
\end{center}
\end{table}
\end{landscape}

\begin{landscape}
\begin{table}[h]
\caption{Results for $B_{\gamma}$ with $G_{12}$}
\label{tab:G12}
\begin{center}
{\fontsize{5.4pt}{7pt}\selectfont
\begin{tabular}{|c||r|r|r|r|r|r|r|r|r|r|r|r|r|r|}
\hline
\hline
\rule[0mm]{0mm}{5mm}&	\multicolumn{2}{c|}{\bf Alg.~1.1~{$(\mathcal{N}_n)$}} &	\multicolumn{2}{c|}{\bf Alg.~1.2~{($\mathcal{H}_n$)}}&	\multicolumn{2}{c|}{\bf Alg.~2.1~{($\mathcal{G}_n$ , $\widehat{\mathcal{G}_n}$)}}& \multicolumn{2}{c|}{\bf Alg.~1.3~{($\mathcal{F}_n^{+}$)}}& \multicolumn{2}{c|}{\bf Alg.~2.2~{($\mathcal{F}_n^{+}$ , $\widehat{\mathcal{F}_n^{+}}$)}}&\multicolumn{2}{c|}{\bf Alg.~2.3~{($\mathcal{F}_n^{\pm}$ , $\widehat{\mathcal{F}_n^{\pm}}$)}}&\multicolumn{2}{c|}{\bf Alg.~1.4~{($\mathcal{S}_n^+$ + $\mathcal{N}_n$)}}\\
\cline{2-15}
\rule[0mm]{0mm}{5mm} $\gamma$&	\multicolumn{1}{c|}{Iter.}&	\multicolumn{1}{c|}{Time(s)}&	\multicolumn{1}{c|}{Iter.}& \multicolumn{1}{c|}{Time(s)}& \multicolumn{1}{c|}{Iter.}&	\multicolumn{1}{c|}{Time(s)}&	\multicolumn{1}{c|}{Iter.}&	\multicolumn{1}{c|}{Time(s)}&	\multicolumn{1}{c|}{Iter.}&	\multicolumn{1}{c|}{Time(s)} &	\multicolumn{1}{c|}{Iter.}&	\multicolumn{1}{c|}{Time(s)}&	\multicolumn{1}{c|}{Iter.}&\multicolumn{1}{c|}{Time(s)}\\
\hline
\hline
2.0&		2111& 0.483&		2111& 0.512&				2111&8.500&					2111&13.782&					2111&26.713&					2111&49.377&					2111&2468.863\\
2.1&		2111& 0.483&		2111& 0.504&				2111&8.561&					2111&13.682&					2111&26.726&					2111&49.184&					2111&2466.339\\
2.2&		2111& 0.482& 	2111& 0.500&				2111&8.495&					2111&13.585&					2111&26.273&					2111&48.592&					2111&2471.254\\
2.3&		2111& 0.483&		2111& 0.502&				2111&8.565&					2111&13.520&					2111&26.261&					2111&48.528&					2111&2461.198\\
2.4&		2111& 0.489&		2111& 0.501&				2111&8.573&					2111&13.315&					2111&25.846&					2111&47.172&					2111&2464.471\\
2.5&		2111& 0.485&		2111& 0.502&				2111&8.558&					2111&13.344&					2111&26.037&					2111&47.616&					2111&2467.267\\
2.6&		2111& 0.481&		2111& 0.501&				2111&8.522&					2111&13.304&					2111&25.958&					2111&47.223&					2111&2464.474\\
2.7&		4097& 1.136&		4097& 1.173&				4097&16.648&					4097&25.753&					4097&49.782&					4097&90.116&					4097&4554.881\\
2.8&		4097& 1.136&		4097& 1.171&				4097&16.752&					4097&25.655&					4097&49.553&					4097&89.579&					4097&4559.754\\
2.9&		4097& 1.134&		4097& 1.168&				4097&16.779&					4097&25.643&					4097&49.263&					4097&89.391&					4097&4557.203\\
3.0&		4097& 1.134&		4097& 1.167&				4097&16.652&					4097&25.550&					4097&48.423&					4097&87.842&					4097&4551.061\\
3.1&		4097& 1.136&		4097& 1.166&				4097&16.751&					4097&25.484&					4097&48.789&					4097&87.737&					4097&4555.928\\
3.2&		4097& 1.138&		4097& 1.175&				4097&16.734&					4097&25.422&					4097&48.855&					4091&87.169&					4090&4548.440\\
3.3&		4097& 1.131&		4095& 1.167&				4097&16.784&					4097&25.274&					4091&48.418&					4089&86.574&					4085&4522.061\\
3.4&		4097& 1.131&		4087& 1.161&				4097&16.969&					4097&25.179&					4091&48.164&					4085&86.034&					4085&4521.749\\
3.5&		4097& 1.139&		4085& 1.164&				4097&16.731&					4087&25.122&					4087&48.091&					4085&85.644&					4085&4539.213\\
3.6&		4097& 1.137&		4085& 1.161&				4091&16.721&					4087&24.920&					4087&47.638&					4085&84.635&					4085&4533.101\\
3.7&		4097& 1.140&		4085& 1.158&				4091&16.755&					4087&24.834&					4085&47.310&					4085&84.454&					4031&4384.409\\
3.8&		4097& 1.133&		4084& 1.162&				4089& 17.128& 				4087&24.831&					4085& 48.094&				4075&85.390& 				4023&4853.335\\
3.9&		4097& 1.137&		4080& 1.187&				4089& 17.144& 				4081&24.719&					4079& 47.219&				4051&84.028& 				4023&5004.118\\
4.0&	 	-&-&			-&	 -&					-& -& 						-& -&						-& - &						-&-& 						-&-\\
4.1&		-&-&			-&	 -&					-& -& 						-& -&						-&- &						827717&18054.273& 			4013&5589.341\\
4.2&		-&-&			-&	 -&					-& -& 						-& -& 						899627&14932.525 &			296637&5093.561&				345&528.262\\
4.3&		-&-&			-&	 -&					-& -& 						1024493& 15985.310&			469665& 6007.219 &			102211&1559.054&				39&50.717\\
4.4&		-&-&			1467851& 16744.884&		-& -& 						592539& 6657.898&			147363& 1361.774 &			36937&545.801&				21&26.4293\\
4.5&		-&-&			1125035& 9820.911&		-& -& 						354083& 3066.114&			66819& 559.987&				14533&213.376&				17&20.961\\
4.6&		-&-&			762931&	5680.756&		1107483& 14991.047&			213485&	1506.465&			25675& 206.767&				4603&69.503&					7&8.768\\
4.7&		-&-&			610071&	4319.490&		793739&	8137.410&			125747&	768.224&				22119& 180.072&				1957&30.490&					3&3.809\\
4.8&		-&-&			569661&	3799.361&		473137&	3413.271&			69887& 386.279&				20997& 176.279&				645&10.347&					3&4.051\\
4.9& 	-&-&			407201&	1834.912&		232295&	1231.091&			39091& 207.440& 				1969& 16.716& 				109&1.889& 					1&1.221\\
5.0&		-&-&			305627&	974.829&			190185&	859.674&				21283& 112.276&				1213& 10.501&				1&0.014&						1&1.189\\
5.1&		-&-&			206949&	415.090&			34641&	113.631&				12165& 64.742&				219& 2.000&					1&0.013&						1&1.150\\
5.2&		-&-&			141383&	172.541&			1& 0.004&					1& 0.008&					1& 0.008&					1&0.012&						1&1.120\\
5.3&		-&-&			110641&	101.475&			1& 0.003&					1& 0.008&					1& 0.007&					1&0.012&						1&1.040\\
5.4&		-&-&			90877& 67.681&			1& 0.003&					1& 0.008&					1& 0.008&					1&0.012&						1&1.078\\
5.5&		-&-&			44731& 14.292&			1& 0.003&					1& 0.007&					1& 0.007& 					1&0.011&						1&1.100\\
5.6&		-&-&			26171& 5.910&			1& 0.004&					1& 0.007&					1& 0.007&					1&0.012&						1&1.000\\
5.7&		-&-&			15045& 2.775&			1& 0.004&					1& 0.008&					1& 0.008&					1&0.012&						1&1.057\\
5.8&		-&-&			10239& 1.705&			1& 0.003&					1& 0.007&					1& 0.007&					1&0.012&						1&1.063\\
5.9&		-&-&			6977&	 1.042&			1& 0.003&					1& 0.007&					1& 0.007&					1&0.011&						1&1.051\\
6.0&		-&-&			4717&	 0.654&			1& 0.006&					1& 0.007&					1& 0.008&					1&0.012&						1&1.119\\
\hline
\hline
\end{tabular}
}
\end{center}
\end{table}

\begin{figure}[h]
\begin{center}
\label{fig:graph G12v}
\includegraphics[width=200mm]{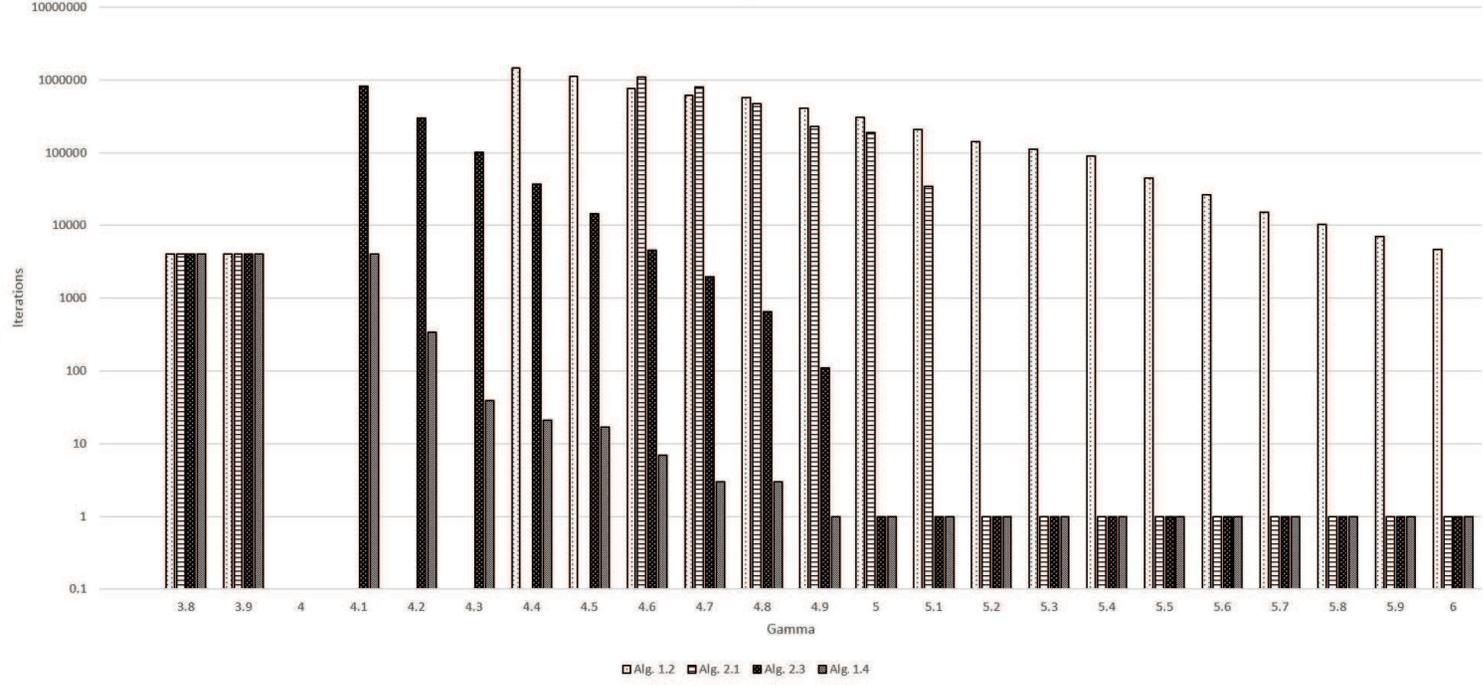}
\end{center}
\caption{Graph of Table \ref{tab:G12}: Iterations vs. $\gamma$ of {\bf Algorithms 1.2, 2.1, 2.3 and 21.4} for the graph $G_{12}$.}
\label{fig:graph G12}
\end{figure}

\begin{figure}[h]
\begin{center}
\label{fig:graph G12v}
\includegraphics[width=200mm]{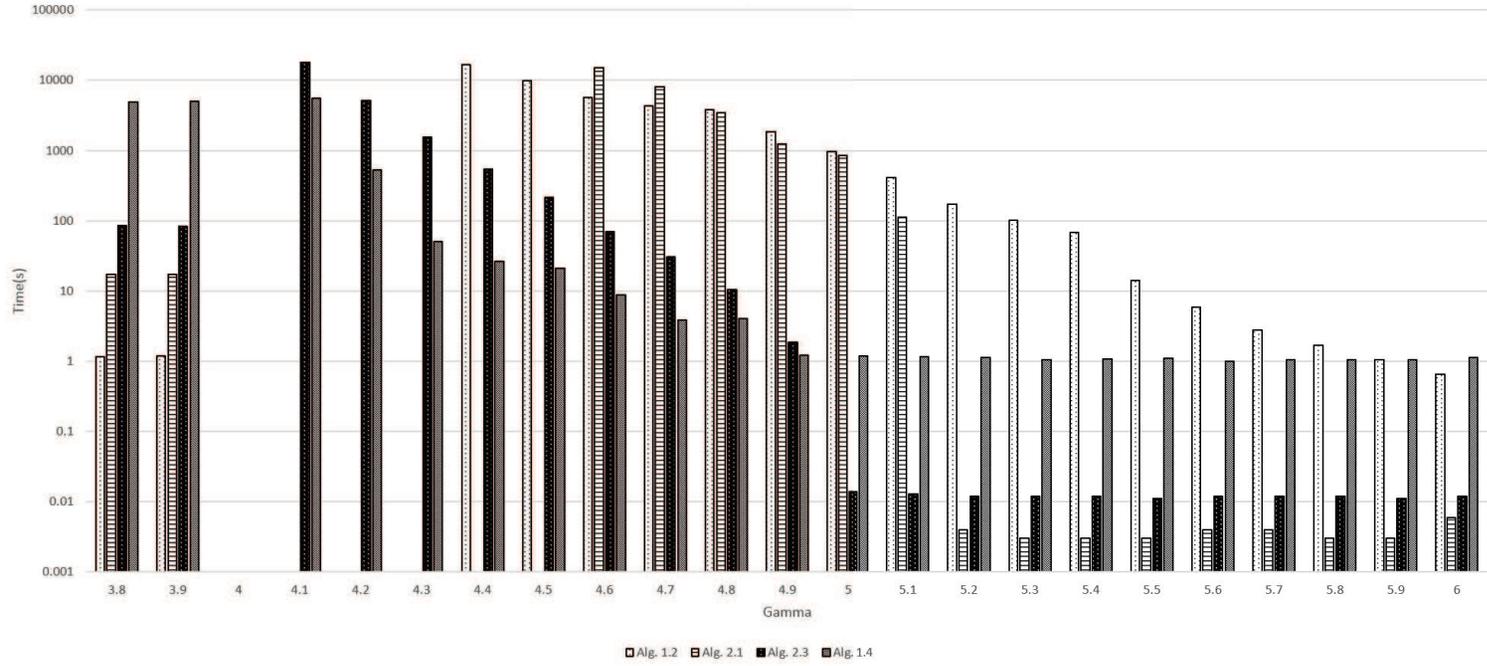}
\end{center}
\caption{Graph of Table \ref{tab:G12}: Time(s) vs.  $\gamma$ of {\bf Algorithms 1.2, 2.1, 2.3 and 21.4} for the graph $G_{12}$.}
\label{fig:graph G12}
\end{figure}

\begin{table}[h]
\caption{Results for $C_{\gamma}$ with $n=10$}
\label{tab:n10}
\begin{center}
{\fontsize{5pt}{4.6pt}\selectfont
\begin{tabular}{|c||r|r|r|r|r|r|r|r|r|r|r|r|r|r|}
\hline
\hline
\rule[0mm]{0mm}{5mm}&	\multicolumn{2}{c|}{\bf Alg.~1.1~{$(\mathcal{N}_n)$}} &	\multicolumn{2}{c|}{\bf Alg.~1.2~{($\mathcal{H}_n$)}}&	\multicolumn{2}{c|}{\bf Alg.~2.1~{($\mathcal{G}_n$ , $\widehat{\mathcal{G}_n}$)}}& \multicolumn{2}{c|}{\bf Alg.~1.3~{($\mathcal{F}_n^{+}$)}}& \multicolumn{2}{c|}{\bf Alg.~2.2~{($\mathcal{F}_n^{+}$ , $\widehat{\mathcal{F}_n^{+}}$)}}&\multicolumn{2}{c|}{\bf Alg.~2.3~{($\mathcal{F}_n^{\pm}$ , $\widehat{\mathcal{F}_n^{\pm}}$)}}&\multicolumn{2}{c|}{\bf Alg.~1.4~{($\mathcal{S}_n^+$ + $\mathcal{N}_n$)}}\\
\cline{2-15}
\rule[0mm]{0mm}{5mm} $\gamma$&	\multicolumn{1}{c|}{Iter.}&	\multicolumn{1}{c|}{Time(s)}&	\multicolumn{1}{c|}{Iter.}& \multicolumn{1}{c|}{Time(s)}& \multicolumn{1}{c|}{Iter.}&	\multicolumn{1}{c|}{Time(s)}&	\multicolumn{1}{c|}{Iter.}&	\multicolumn{1}{c|}{Time(s)}&	\multicolumn{1}{c|}{Iter.}&	\multicolumn{1}{c|}{Time(s)} &	\multicolumn{1}{c|}{Iter.}&	\multicolumn{1}{c|}{Time(s)}&	\multicolumn{1}{c|}{Iter.}&\multicolumn{1}{c|}{Time(s)}\\
\hline
\hline
-5.00000&		 		1&0.001&			1&	 0.001&					1& 0.009& 						1& 0.009&						1& 0.016&						1&0.025& 						1&0.501\\
-7.50000&	 		1&0.001&			1&	 0.001&					1& 0.010& 						1& 0.009&						1& 0.016&						1&0.025& 						1&0.491\\
-8.75000&	 		1&0.001&			1&	 0.001&					1& 0.010& 						1& 0.008&						1& 0.016&						1&0.025& 						1&0.481\\
-9.37500&	 		2&.001&			2&	 0.001&					2& 0.020& 						2& 0.018&						2& 0.032&						2&0.048& 						2&0.946\\
-9.68750&		 	3&0.001&			3&	 0.001&					3& 0.021& 						3& 0.021&						3& 0.041&						3&0.066& 						3&1.621\\
-9.84375&	 	799&0.267&		799& 0.259&					799& 4.993& 						799& 4.815&						799& 9.253&						799&15.329& 						799&441.790\\
-10.00000&	 		-&-&			-&	 -&						-& -& 							-& -&							-& - &							-&-& 							-&-\\
-10.15625&	 	84828&81.527&	1& 0.001&					1& 0.029& 						1& 0.007&						1& 0.011&						1&0.012& 						1&0.424\\
-10.31250&	 	39790&17.030&	1& 0.001&					1& 0.004& 						1& 0.008&						1& 0.011&						1&0.012& 						1&0.362\\
-10.62500&	 		5546&1.210&		1& 0.001&					1& 0.005& 						1& 0.008&						1& 0.011&						1&0.012& 						1&0.380\\
-11.25000&	 		8&0.007&			1& 0.001&					1& 0.005& 						1& 0.007&						1& 0.011&						1&0.012& 						1&0.363\\
-12.50000&	 		2&0.001&			1& 0.001&					1& 0.005& 						1& 0.007&						1& 0.011&						1&0.012& 						1&0.367\\
-15.00000&	 		2&0.001&			1& 0.001&					1& 0.005& 						1& 0.008&						1& 0.011&						1&0.012& 						1&0.384\\
\hline
\hline
\end{tabular}
}
\end{center}
\end{table}


\begin{table}[h]
\caption{Results for $C_{\gamma}$ with $n=20$}
\label{tab:n20}
\begin{center}
{\fontsize{5pt}{4.6pt}\selectfont
\begin{tabular}{|c||r|r|r|r|r|r|r|r|r|r|r|r|r|r|}
\hline
\hline
\rule[0mm]{0mm}{5mm}&	\multicolumn{2}{c|}{\bf Alg.~1.1~{$(\mathcal{N}_n)$}} &	\multicolumn{2}{c|}{\bf Alg.~1.2~{($\mathcal{H}_n$)}}&	\multicolumn{2}{c|}{\bf Alg.~2.1~{($\mathcal{G}_n$ , $\widehat{\mathcal{G}_n}$)}}& \multicolumn{2}{c|}{\bf Alg.~1.3~{($\mathcal{F}_n^{+}$)}}& \multicolumn{2}{c|}{\bf Alg.~2.2~{($\mathcal{F}_n^{+}$ , $\widehat{\mathcal{F}_n^{+}}$)}}&\multicolumn{2}{c|}{\bf Alg.~2.3~{($\mathcal{F}_n^{\pm}$ , $\widehat{\mathcal{F}_n^{\pm}}$)}}&\multicolumn{2}{c|}{\bf Alg.~1.4~{($\mathcal{S}_n^+$ + $\mathcal{N}_n$)}}\\
\cline{2-15}
\rule[0mm]{0mm}{5mm} $\gamma$&	\multicolumn{1}{c|}{Iter.}&	\multicolumn{1}{c|}{Time(s)}&	\multicolumn{1}{c|}{Iter.}& \multicolumn{1}{c|}{Time(s)}& \multicolumn{1}{c|}{Iter.}&	\multicolumn{1}{c|}{Time(s)}&	\multicolumn{1}{c|}{Iter.}&	\multicolumn{1}{c|}{Time(s)}&	\multicolumn{1}{c|}{Iter.}&	\multicolumn{1}{c|}{Time(s)} &	\multicolumn{1}{c|}{Iter.}&	\multicolumn{1}{c|}{Time(s)}&	\multicolumn{1}{c|}{Iter.}&\multicolumn{1}{c|}{Time(s)}\\
\hline
\hline
-5.00000&		 	1&0.002&			1&	 0.002&				1& 0.042& 					1& 0.095&					1& 0.181&					1&0.182& 					1&13.812\\
-7.50000&	 		1&0.002&			1&	 0.002&				1& 0.045& 					1& 0.093&					1& 0.183&					1&0.180& 					1&14.266\\
-8.75000&	 		2&0.005&			2&	 0.005&				2& 0.046& 					2& 0.191&					2& 0.368&					2& 0.378& 					2&25.224\\
-9.37500&	 		128&0.091&		128&	0.096&			128& 2.490& 					128& 11.755&					128& 22.660&					128& 22.682& 				128&1598.031\\
-9.68750&		 	-&-&			-&	 -&					-& -& 						-& -&						-& - &						-&-& 						-&-\\
-9.84375&	 		-&-&			-&	 -&					-& -& 						-& -&						-& - &						-&-& 						-&-\\
-10.00000&	 		-&-&			-&	 -&					-& -& 						-& -&						-& - &						-&-& 						-&-\\
-10.15625&	 		-&-&			-&	 -&					-& -& 						1& 0.080&					1& 0.080&					1& 0.076& 					1&13.912\\
-10.31250&	 		-&-&			-&	 -&					1& 0.010& 					1& 0.072&					1& 0.072&					1& 0.073& 					1&14.015\\
-10.62500&	 		-&-&			2&	 0.001&				1& 0.009& 					1& 0.072&					1& 0.069&					1& 0.073& 					1&14.135\\
-11.25000&	 		-&-&			2&	 0.001&				1& 0.009& 					1& 0.067&					1& 0.067&					1& 0.071& 					1&12.126\\
-12.50000&	 		3187&1.370&		2&	 0.001&				1& 0.010& 					1& 0.068&					1& 0.066&					1& 0.067& 					1&11.128\\
-15.00000&	 		2&0.001&			2&	 0.001&				1& 0.010& 					1& 0.069&					1& 0.080&					1& 0.068& 					1&9.580\\
\hline
\hline
\end{tabular}
}
\end{center}
\end{table}
\end{landscape}
\end{document}